\documentclass[reqno, 12pt,a4letter]{amsart}
\usepackage{amsmath,amsxtra,amssymb,latexsym, amscd,amsthm}
\usepackage[utf8]{inputenc}
\usepackage[mathscr]{euscript}
\usepackage{mathrsfs}

\usepackage[english]{babel}
\usepackage[active]{srcltx}
\usepackage{color}

\newtheorem {theorem}{Theorem}[section]

\newtheorem {corollary}{Corollary}[section]
\newtheorem {proposition}{Proposition}[section]
\newtheorem {lemma}{Lemma}[section]

\newtheorem {definition}{Definition}[section]
\newtheorem {remark}{Remark}[section]

\newcommand{\R}{\mathbb{R}}

\def\tr{\mathrm{tr}}

\def\co{\mathrm{co}}

\def\EES{{\accent"5E e}\kern-.5em\raise.8ex\hbox{\char'23 }}
\def\ow{o\kern-.42em\raise.82ex\hbox{
   \vrule width .12em height .0ex depth .075ex \kern-0.16em \char'56}\kern-.07em}
\def\OW{o\kern-.460em\raise1.36ex\hbox{
\vrule width .13em height .0ex depth .075ex \kern-0.16em
\char'56}\kern-.07em}
\def\DD{D\kern-.7em\raise0.4ex\hbox{\char '55}\kern.33em}

\title{\L ojasiewicz-type inequalities with explicit exponents for the largest eigenvalue function of real symmetric polynomial matrices}

\author{S\~i Ti\d{\^e}p \DD INH$^\dagger$}
\address{$^\dagger$Institute of Mathematics, VAST, 18, Hoang Quoc Viet Road, Cau Giay District 10307, Hanoi, Vietnam}
\email{dstiep@math.ac.vn}

\author{Ti\EES n S\OW n Ph\d{A}m$^\ddagger$}
\address{$^\ddagger$Department of Mathematics, University of Dalat, 1 Phu Dong Thien Vuong, Dalat, Vietnam}
\email{sonpt@dlu.edu.vn}

\thanks{$^\dagger$This author's research is funded by Vietnam National Foundation for Science and Technology Development (NAFOSTED) under grant number 101.04-2014.23 and the Vietnam Academy of Science and Technology (VAST)}

\thanks{$^\ddagger$This author's research is funded by Vietnam National Foundation for Science and Technology Development (NAFOSTED) under grant number 101.04-2013.07.}

\subjclass{Primary 32B20; Secondary 14P10}

\keywords{Largest eigenvalue; \L ojasiewicz inequalities; non-degeneracy; polynomial matrices}

\date{ \today}

\begin{document}
\maketitle

\begin{abstract}
Let $F(x) := (f_{ij}(x))_{i,j=1,\ldots,p},$ be a real symmetric  polynomial matrix of order $p$ and let $f(x)$ be the largest eigenvalue function of the matrix $F(x).$
We denote by ${\partial}^\circ f(x)$ the Clarke subdifferential of $f$ at $x.$
In this paper, we first give the following {\em nonsmooth} version of \L ojasiewicz gradient inequality for the function $f$ with an explicit exponent: For any $\bar x\in \Bbb R^n$ there exist $c > 0$ and $\epsilon > 0$ such that we have for all $\|x - \bar{x}\| < \epsilon,$
\begin{equation*}
\inf \{ \| w \| \ : \ w \in {\partial}^\circ f(x) \} \ \ge \ c\, |f(x) - f(\bar x)|^{1 - \frac{1}{\mathscr{R}(2n+p(n+1),d+3)}},
\end{equation*}
where $d:=\max_{i,j = 1, \ldots, p}\deg f_{i j}$ and  $\mathscr{R}$ is a function introduced  by D'Acunto and Kurdyka: $\mathscr{R}(n, d) := d(3d - 3)^{n-1}$ if $d \ge 2$ and $\mathscr{R}(n, d) := 1$ if $d = 1.$ Then we establish some local and global versions of \L ojasiewicz inequalities which bound the distance function to the set $\{x\in\Bbb R^n\ : \ f(x)\le 0\}$ by some exponents of the function $[f(x)]_+:=\max\{f(x),0\}$.
\end{abstract}

\pagestyle{plain}

\section{Introduction}

Given an extend real-valued lower semicontinuous function $f:\Bbb R^n\to\Bbb R\cup\{+\infty\}$ and a set $K \subset \mathbb{R}^n$, consider the set $S$ given by the following inequality
\begin{equation}\label{EQ1}
f(x)\le 0, \quad x\in K.
\end{equation}
Let $x \in K$ be such that $f(x)>0$. In general, it is hard to answer the following questions: Based on the value of $f$ at $x$, how close is $x$ to $S$? In other words, if $f(x)$ is small, whenever is $x$ a good approximation of a point in $S$, i.e., the distance from $x$ to $S$ is small? However, in many cases, these questions can be answered by bounding the distance to $S$ by some exponents of the function $[f(x)]_+:=\max\{f(x),0\}$. In this paper, when $K$ is compact, we are interested in \L ojasiewicz-type inequalities of the following forms

\begin{equation}
\label{LEB}c\, \mathrm{dist}(x, S)\le [f(x)]_+^\alpha \quad \text{ for all } \quad x\in K,
\end{equation}
\begin{equation}
\label{GEB}c\, \mathrm{dist}(x, S)\le [f(x)]_+^\alpha+[f(x)]_+^\beta \quad \text{ for all } \quad x\in \Bbb R^n,
\end{equation}
where $c>0$ is a constant, $\alpha>0,\ \beta>0$ are some constants to be determined and $\mathrm{dist}(x,S)$ is the Euclidean distance from $x$ to $S$.

Let $V := \{x \in \mathbb{R}^n \ : \ f(x) = 0\}$. When the function $f$ is real analytic, the existence of inequality (\ref{LEB}) can be deduced easily from the following (classical) \L ojasiewicz inequality by noting that $\mathrm{dist}(x,V)=\mathrm{dist}(x,S)$ for $f(x)\ge 0$.

\begin{theorem} [see \cite{Hormander1958, Lojasiewicz1958, Lojasiewicz1959, Lojasiewicz1965}] 
Assume that $f^{-1}(0) \ne \emptyset$ and let $K$ be a compact subset in $\Bbb{R}^n$. Then there exist $c>0$ and $\alpha > 0$ such that
$$c\, \mathrm{dist}(x,V)^\alpha \le |f(x)|,  \quad \mbox{ for } \quad x \in K.$$
\end{theorem}

When $K$ is not compact, inequality (\ref{LEB}) does not hold in general. However, when $K$ is defined by a system of polynomial inequalities $K:=\{f_1(x)\ge 0,\ldots,f_p(x)\ge 0\}$, under some assumptions of non-degeneracy at infinity, the authors of \cite{HaHV2013} (for $p=1$) and of  \cite{Dinh2014-1} proved the existence of inequality (\ref{GEB}) where $f(x) := \max_{i=1,\ldots,p}f_i(x)$ and $c,\alpha,\beta$ are some positive constants. Moreover, the exponents $\alpha,\beta$ are determined explicitly.

In this paper, we restrain to the case that $f$ is the largest eigenvalue function of a symmetric polynomial matrix. Sensitivity results on eigenvalue functions are important in view of applications. Largest eigenvalue or matrix norm minimization arises in control theory, structural and combinatorial optimization, graph theory, stability analysis of dynamic systems etc. We invite the reader to the survey \cite{Lewis1996} for more details.

Here and in the following, ${\mathscr{R}}$ is a function defined by:
\begin{equation}\label{Eq4}
\mathscr{R}(n, d) := \begin{cases}
d(3d - 3)^{n-1} & \text { if } d\ge 2,\\
1 & \text { if } d=1,
\end{cases}
\end{equation}
for any positive integers $n$ and $d.$ Let $\mathbb{B}^n(x, r)$ denote the closed ball of radius $r$ centered at $x$ , let $\mathbb{B}^n$ and $\mathbb{S}^{n - 1}$ be the closed unit ball and the unit sphere, respectively. For each real number $r,$ we put $[r]_+ := \max\{r, 0\}.$ 
We denote by ${\mathcal S}^p$ the set of real symmetric matrices of order $p$. We write $A \succeq 0$ (resp., $A \preceq 0$) if $A\in {\mathcal S}^p$ is positive (resp., negative) semidefinite. The trace of a symmetric matrix $A \in {\mathcal S}^p$ is denoted by $\tr(A).$

The first main result of the paper is a nonsmooth version of \L ojasiewicz gradient inequality for the largest eigenvalue function with an explicit exponent, which is an important tool to prove the existence of (\ref{LEB}). The estimation of the exponent is based on the estimation of \L ojasiewicz exponent in the \L ojasiewicz gradient inequality for polynomials given by D'Acunto and Kurdyka in \cite{Acunto2005}.  Now with the definitions in the next section, the first main contribution of this paper is the following.

\begin{theorem}\label{NonSmoothTheorem}
Let $F \colon  \Bbb R^n \rightarrow {\mathcal S}^p,\ x\mapsto F(x)=(f_{ij}(x))_{i,j=1,\ldots,p},$ be a symmetric polynomial matrix of order $p$. Let $f(x)$ be the largest eigenvalue function and $d:=\max_{i,j = 1, \ldots, p}\deg f_{i j}.$ Then for any $\bar x\in \Bbb R^n,$ there exist $c > 0$ and $\epsilon > 0$ such that we have for all $x \in \mathbb{B}^n(\bar{x}, \epsilon),$
\begin{equation}\label{Eq5}
\inf \{ \| w \| \ : \ w \in {\partial}^\circ f(x) \} \ \ge \ c\, |f(x) - f(\bar x)|^{1 - \frac{1}{\mathscr{R}(2n+p(n+1),d+3)}}.
\end{equation}
In particular,
\begin{equation}\label{Eq5b}
\frak m_f(x) \ \ge \ c\, |f(x) - f(\bar x)|^{1 - \frac{1}{\mathscr{R}(2n+p(n+1),d+3)}} \quad \textrm{ for all } \quad x \in \mathbb{B}^n(\bar{x}, \epsilon),
\end{equation}
where $\frak m_f$ is  the nonsmooth slope of $f$ (Definition \ref{mf}).
\end{theorem}

As applications, we prove a local \L ojasiewicz-type inequality (Theorem \ref{CompactErrorBound}) and a version of separation of semialgebraic sets with explicit exponents (Proposition \ref{Separation}). 

In global context, we give two versions of global \L ojasiewicz-type inequalities with explicit exponents, one is obtained by modifying the left side of (\ref{LEB}) by dividing this side by an explicit function which is big ``at infinity" (Corollary \ref{GlobalKollar}), the other takes the form of (\ref{GEB}). Precisely, inspired by \cite{Khovanskii1978} and \cite{Kouchnirenko1976}, we introduce a new condition of non-degeneracy at infinity for symmetric polynomial matrices under which, we study global \L ojasiewicz-type inequality of the type (\ref{GEB}) where $f$ is the largest eigenvalue function of a symmetric polynomial matrix.

\begin{theorem}\label{HolderTypeTheorem}
Let $F \colon \Bbb R^n \rightarrow {\mathcal S}^p,\ x\mapsto F(x)=(f_{ij}(x))_{i,j=1,\ldots,p},$ be a symmetric polynomial matrix of order $p$ such that $S_F := \{x \in {\Bbb R}^n \ : \ F(x) \preceq  0\} \ne \emptyset.$ Suppose that $F$ is non-degenerate at infinity, $f_{ii}$ is convenient, and that $\Gamma(f_{ij})\subseteq\Gamma(f_{ii})$ for $i,j=1,\ldots,p.$ Let $f(x)$ be the largest eigenvalue function. Then there exists a constant $c>0$ such that
\begin{equation*}
c\, \mathrm{dist}(x, S_F) \ \le \ [f(x)]_+^{\frac{1}{\mathscr{R}(2n+p(n + 1), d+3)}} + [f(x)]_+ \quad \textrm{ for all } \quad x \in \mathbb{R}^n,
\end{equation*}
where $d:=\max_{i,j = 1, \ldots, p}\deg f_{i j}.$
\end{theorem}

The paper is structured as follows. Section \ref{SectionPreliminary} presents some backgrounds in semi-algebraic geometry, subdifferentials, nonsmooth slope and Newton polyhedra. Theorem \ref{NonSmoothTheorem} is proved in Section \ref{LojasiewiczGradient}. Sections \ref{LocalErrorBounds} and \ref{GlobalErrorBounds} present some consequences of Theorem \ref{NonSmoothTheorem}. The proof of Theorem \ref{HolderTypeTheorem} is given in Section~\ref{GlobalHolderErrorBounds}. 

\section{Preliminaries} \label{SectionPreliminary}

\subsection{Semi-algebraic geometry}
In this subsection, we recall some notions and results of semi-algebraic geometry, which can be found in \cite{Benedetti1991, Bierstone1988, Bochnak1998, Dries1996}.

\begin{definition}{\rm
\begin{enumerate}
  \item[(i)] A subset of $\R^n$ is called {\em semi-algebraic} if it is a finite union of sets of the form
$$\{x \in\R^n \ : \ f_i(x) = 0, i = 1, \ldots, k; f_i(x) > 0, i = k + 1, \ldots, p\}$$
where all $f_{i}$ are polynomials.
 \item[(ii)]
Let $A \subset \Bbb{R}^n$ and $B \subset \Bbb{R}^p$ be semi-algebraic sets. A map $F \colon A \to B$ is said to be {\em semi-algebraic} if its graph
$$\{(x, y) \in A \times B \ : \ y = F(x)\}$$
is a semi-algebraic subset in $\Bbb{R}^n\times\Bbb{R}^p.$
\end{enumerate}
}\end{definition}

A major fact concerning the class of semi-algebraic sets is its stability under linear projections (see \cite{Seidenberg1954, Tarski1931, Tarski1951}).

\begin{theorem}[Tarski--Seidenberg Theorem] \label{TarskiSeidenbergTheorem}
The image of a semi-algebraic set by a semi-algebraic map is semi-algebraic.
\end{theorem}

We list below some basic properties of semi-algebraic sets and functions.
\begin{enumerate}
\item[(i)] The class of semi-algebraic sets is closed with respect to Boolean operators; a Cartesian product of semi-algebraic sets is a semi-algebraic set;

\item[(ii)] The closure and the interior of semi-algebraic sets are semi-algebraic sets;

\item[(iii)]  A composition of semi-algebraic maps is a semi-algebraic map;

\item[(iv)]  If $S$ is a semi-algebraic set, then the distance function $$\mathrm{dist}(\cdot, S) \colon {\Bbb R}^n \rightarrow {\Bbb R}, \quad x \mapsto \mathrm{dist}(x, S) := \inf \{\|x - a\| \ : \ a \in S \},$$
is also semi-algebraic.
\end{enumerate}

Now we give a version of Curve Selection Lemma which will be used in the proof of Theorem \ref{HolderTypeTheorem}. For more details, see \cite{Milnor1968,Nemethi1992} and see \cite{Dinh2013} for a complete proof.

\begin{lemma}[Curve Selection Lemma at infinity]\label{CurveSelectionLemma}
Let $A\subset \mathbb{R}^n$ be a semi-algebraic set, and let $F := (f_1, \ldots,f_p) \colon  \mathbb{R}^n \to \mathbb{R}^p$ be a semi-algebraic map. Assume that there exists a sequence $x^k\in A$ such that $\lim_{k \to \infty} \| x^k  \| = \infty$ and $\lim_{k \to \infty} F(x^k)  = y \in(\overline{\mathbb{R}})^p,$ where $\overline{\mathbb{R}} := \mathbb{R} \cup \{\pm \infty\}.$ Then there exists a smooth semi-algebraic curve $\varphi \colon (0, \epsilon)\to \mathbb{R}^n$ such that $\varphi(t) \in A$ for all $t \in (0, \epsilon), \lim_{t \to 0} \|\varphi(t)\| = \infty,$ and $\lim_{t \to 0} F(\varphi(t)) = y.$
\end{lemma}

The following Growth Dichotomy Lemma is also useful in the proof of Theorem \ref{HolderTypeTheorem} (see, e.g., \cite{Dries1996, Miller1994}).

\begin{lemma}[Growth Dichotomy Lemma] \label{GrowthDichotomyLemma}
Let $f \colon (0, \epsilon) \rightarrow {\Bbb R}$ be a semi-algebraic function with $f(t) \ne 0$ for all $t \in (0, \epsilon).$ Then
there exist some constants $c \ne 0$ and $q \in {\Bbb Q}$ such that $f(t) = ct^q + o(t^q)$ as $t \to 0^+.$
\end{lemma}

To end this part, let us recall the following \L ojasiewicz gradient inequality with an explicit exponent which will be used in the proof of Theorem \ref{NonSmoothTheorem}.

\begin{theorem} [see \cite{Acunto2005}] \label{GradientInequality}
Let $f \colon \mathbb{R}^n \rightarrow \mathbb{R}$ be a polynomial function of degree $d$. Assume that $f(\bar x)=0.$ Then there are some positive constants $c$ and $\epsilon$ such that
$$\|\nabla f(x)\| \ge c\, |f(x)|^{1 - \frac{1}{\mathscr{R}(n, d)}} \quad \text{ for all } \quad \|x - \bar{x}\| \le \epsilon,$$
where $\mathscr{R}(n,d)$ is defined by (\ref{Eq4}).
\end{theorem}

\subsection{Subdifferentials and nonsmooth slope}
We first recall some notions of subdifferential, which are crucial for our considerations. For nonsmooth analysis we refer to the comprehensive texts \cite{Clarke1983, Mordukhovich2006, Rockafellar1998}.

\begin{definition}{\rm
Let $f \colon {\Bbb R}^n \rightarrow {\Bbb R}$ be a continuous function. For any $x \in {\Bbb R}^n$ let us define
\begin{enumerate}
  \item[(i)] The {\em Fr\'echet subdifferential} $\hat{\partial} f(x)$ of $f$ at $x \in {\Bbb R}^n$:
$$\hat{\partial} f(x) := \left \{ v \in {\Bbb R}^n \ : \ \liminf_{\| h \| \to 0, \ h \ne 0} \frac{f(x + h) - f(x) - \langle v, h \rangle}{\| h \|} \ge 0 \right \}.$$
  \item[(ii)] The {\em limiting subdifferential} ${\partial} f(x)$ of $f$ at $x$ is the set of all cluster points of sequences $\{v^k\}_{k \ge 1}$ such that $v^k\in \hat{\partial} f(x^k)$ and $(x^k, f(x^k)) \to (x, f(x))$ as $k \to \infty.$
  \item[(iii)] Assume that $f$ is locally Lipschitz. By Rademacher's theorem, $f$ has at almost all points $x \in \mathbb{R}^n$ a gradient, which we denote $\nabla f(x).$ Then the {\em Clarke subdifferential} ${\partial}^\circ f(x)$ of $f$ at $x$ is defined by
$${\partial}^\circ f(x) := {\co}\{\lim \nabla f(x^k) : x^k\to x\},$$
where $\mathrm{co}(A)$ stands for the convex hull of a set $A.$
\end{enumerate}
}\end{definition}

\begin{remark}{\rm
\begin{enumerate}
  \item[(i)] It is a well-known result of variational analysis that $\hat{\partial} f(x)$ (and a fortiori $\partial f(x)$ and $\partial^\circ f(x)$) is not empty in a dense subset of the domain of $f$ (see e.g., \cite{Rockafellar1998}).

  \item[(ii)] From the above definitions, it follows clearly that for all $x \in \mathbb{R}^n,$ one has
$$\hat{\partial}f(x)\subset {\partial}f(x).$$
\item[(iii)]  If $f$ is differentiable around $x$, then we have
$${\partial}^\circ f(x) = {\partial} f(x) = \{\nabla f(x)\}.$$

\item[(iv)]  If $f$ is locally Lipschitz, the valued-set mapping ${\Bbb R}^n \rightrightarrows {\Bbb R}^n, x \mapsto {\partial}^\circ  f(x),$ is bounded on compact subsets of ${\Bbb R}^n$ and ${\partial}^\circ  f(x) = \overline{\mathrm{\co}} {\partial} f(x)$ (see e.g., \cite[Theorem 2]{Ioffe1984}).
\end{enumerate}
}\end{remark}

\begin{definition}\label{mf}{\rm
Using the limiting subdifferential $\partial f,$ we define the {\em nonsmooth slope} of $f$ by
$${\frak m}_f(x) := \inf \{ \| w \| \ : \ w \in {\partial} f(x) \}.$$
By definition, ${\frak m}_f(x) = + \infty$ whenever ${\partial} f(x) = \emptyset.$
}\end{definition}

\begin{remark}{\rm
By Tarski--Seidenberg Theorem \ref{TarskiSeidenbergTheorem}, it is not hard to show that if the function $f$ is semi-algebraic then so is ${\frak m}_f.$
}\end{remark}

The following lemma is crucial in the proof of our results since it permits to describe the Clarke subdifferential of the largest eigenvalue function.

\begin{lemma} \label{Lemma1}
Let $F \colon \Bbb R^n \rightarrow {\mathcal S}^p,\ x\mapsto F(x)=(f_{ij}(x))_{i,j=1,\ldots,p},$ be a symmetric polynomial matrix of order $p$ and let $f(x)$ be the largest eigenvalue of the matrix $F(x).$ Then the following statements hold
\begin{enumerate}
\item [(i)] $f(x) = \max_{\|v\|=1}\langle F(x)v,v\rangle$ for all $x \in \mathbb{R}^n.$
\item [(ii)] The function $f \colon \mathbb{R}^n \rightarrow \mathbb{R}, x \mapsto f(x),$ is locally Lipschitz.
\item [(iii)] The Clarke subdifferential $\partial^\circ f(x)$ at $x$ is given by
$$\mathrm{co}\Big \{\nabla_x\langle F(x)v,v\rangle \ : \ v \text{ is a unit eigenvector corresponding to } f(x) \Big \}.$$
More precisely, we have
\begin{eqnarray*}
  \partial^\circ f(x) & = & \Big \{\displaystyle\sum_{l=1}^r\lambda_l\nabla_x\langle F(x)v^l,v^l\rangle
\ : \ r\le n+1,\ \sum_{l=1}^r\lambda_l=1,\ \lambda_l\ge 0, \\
&& \quad v^1,\ldots,v^r \text{ are unit eigenvectors corresponding to } f(x) \Big \}.
\end{eqnarray*}
\end{enumerate}
\end{lemma}
\begin{proof}
(i) is straightforward, and (ii) is a direct consequence of (i) and \cite[Theorem 2.1]{Clarke1975} (see also \cite{Bronstein1979, Kurdyka2008, Wakabayashi1986}).

(iii) The first statement is an immediate consequence of \cite[Theorem~2.1]{Clarke1975}. The second follows form the first and Carath\'eodory's theorem \cite{Caratheodory1911} which says that if a point $z$ belongs to the convex hull $\co(A)$ of a set $A \subset\Bbb R^n$, then $z\in \co(B)$ for some $B \subset A$ and $\mathrm{card} (B) \le n + 1,$ where $\mathrm{card} (B)$ denotes the cardinal of $B$.
\end{proof}

\subsection{Newton polyhedra}

In many problems, the combinatorial information of polynomial maps are important and can be found in their Newton polyhedra. In this subsection, we recall the definition of Newton polyhedra.

Let us begin with some notations which will be used throughout this work. We consider a fixed coordinate system $x_1, \ldots, x_n \in {\Bbb R}^n.$ Let $J \subset \{1, \ldots, n\},$ then we define
$${\Bbb R}^J := \{x \in {\Bbb R}^n \ : \ x_j = 0, \textrm{ for all } j \not \in J\}.$$

We denote by ${\Bbb R}_{\ge 0}$ the set of non-negative real numbers. We also set ${\Bbb Z}_{\ge 0} := {\Bbb R}_{\ge 0} \cap {\Bbb Z}.$ If $\kappa = (\kappa_1, \ldots, \kappa_n) \in {\Bbb Z}_{\ge 0}^n,$ we denote by $x^\kappa$ the monomial
$x_1^{\kappa_1} \cdots x_n^{\kappa_n}$ and by $| \kappa|$ the sum $\kappa_1 + \cdots + \kappa_n.$

\begin{definition} {\rm
A subset $\Gamma \subset {\Bbb R}^n_{\ge 0}$ is said to be a {\em Newton polyhedron at infinity,}
if there exists some finite subset $A \subset {\Bbb Z}^n_{\ge 0}$  such that $\Gamma$ is equal to the convex hull in ${\Bbb R}^n$ of $A \cup \{0\}.$ Then we say that $\Gamma$ is the Newton polyhedron at infinity determined by $A$ and we write $\Gamma = \Gamma(A).$ We say that a Newton polyhedron at infinity $\Gamma \subset {\Bbb R}^n_{\ge 0}$ is {\em convenient} if it intersects each coordinate axis at a point different from the origin, that is, if for any $s \in \{1, \ldots, n\}$  there exists some integer $m_s > 0$ such that $m_s e_s \in \Gamma,$ where $\{e_1, \ldots, e_n\}$ denotes the canonical basis in ${\Bbb R}^n.$
}\end{definition}

Given a Newton polyhedron at infinity $\Gamma \subset {\Bbb R}^n_{\ge 0}$ and a vector $q \in {\Bbb R}^n,$ we define
\begin{eqnarray*}
d(q, \Gamma) &:=& \min \{\langle q, \kappa \rangle \ : \ \kappa \in \Gamma\}, \\
\Delta(q, \Gamma) &:=& \{\kappa \in \Gamma \ : \ \langle q, \kappa \rangle = d(q, \Gamma) \}.
\end{eqnarray*}
We say that a subset $\Delta$ of $\Gamma$ is a {\em face} of $\Gamma$ if there exists a vector $q \in {\Bbb R}^n$ such that $\Delta = \Delta(q, \Gamma).$ The dimension of a face $\Delta$ is defined as the minimum of the dimensions of the affine subspaces containing $\Delta.$ The faces of $\Gamma$ of dimension $0$ are called the {\em vertices} of $\Gamma.$ We denote by $\Gamma_\infty$ the set of the faces of $\Gamma$ which do not contain the origin $0$ in ${\Bbb R}^n.$

\begin{remark}{\rm
By definition, for each face $\Delta$ of $\Gamma_\infty$ there exists a vector $q = (q_1, \ldots, q_n) \in {\Bbb R}^n$ with $\min_{j = 1, \ldots, n} q_j < 0$ such that $\Delta = \Delta(q, \Gamma).$
}\end{remark}

Let $\Gamma_1, \ldots, \Gamma_p$ be a collection of $p$ Newton polyhedra at infinity in ${\Bbb R}^n_{\ge 0},$ for some $p \ge 1.$ The {\em Minkowski
sum} of $\Gamma_1, \ldots, \Gamma_p$ is defined as the set
$$\Gamma_1+ \cdots + \Gamma_p = \{\kappa^1 + \cdots + \kappa^p \ : \ \kappa^i \in \Gamma_i, \textrm{ for all } i = 1, \ldots, p\}.$$
By definition, $\Gamma_1+ \cdots + \Gamma_p$ is again a Newton polyhedron at infinity. Moreover, by applying the definitions given above, it is easy to check that
\begin{eqnarray*}
d(q, \Gamma_1 + \cdots + \Gamma_p) &=& d(q, \Gamma_1) + \cdots + d(q, \Gamma_p), \\
\Delta(q, \Gamma_1 + \cdots + \Gamma_p) &=& \Delta(q, \Gamma_1) + \cdots + \Delta(q, \Gamma_p),
\end{eqnarray*}
for all $q \in \mathbb{R}^n.$ As an application of these relations, we obtain the following lemma whose proof can be found in \cite{Dinh2013}.

\begin{lemma} \label{Lemma2}
{\rm (i)} Assume that $\Gamma$ is a convenient Newton polyhedron at infinity. Let $\Delta$ be a face of $\Gamma$ and let $q=(q_1, \ldots, q_n)\in{\Bbb R}^n$ such that $\Delta = \Delta(q, \Gamma).$ Then the following conditions are equivalent:
\begin{enumerate}
\item[{\rm (i1)}] $\Delta\in\Gamma_\infty$;
\item[{\rm (i2)}] $d(q,\Gamma)<0$;
\item[{\rm (i3)}] $\min_{j = 1, \ldots, n} q_j < 0$.
\end{enumerate}
{\rm (ii)} Assume that $\Gamma_1,\ldots, \Gamma_p$ are some Newton polyhedra at infinity. Let $\Delta$ be a face of the Minkowski sum $\Gamma := \Gamma_1+ \cdots + \Gamma_p.$ Then the following statements hold:
\begin{enumerate}
  \item[{\rm (ii1)}] There exists a unique collection of faces $\Delta_1, \ldots, \Delta_p$ of $\Gamma_1, \ldots, \Gamma_p,$ respectively, such that $$\Delta = \Delta_1 + \cdots + \Delta_p.$$
  \item[{\rm (ii2)}] If $\Gamma_1,\ldots, \Gamma_p$ are convenient, then $\Gamma_\infty \subset \Gamma_{1, \infty}+ \cdots + \Gamma_{p, \infty}.$
\end{enumerate}
\end{lemma}

Let $f \colon {\Bbb R}^n \to {\Bbb R}$ be a polynomial function. Suppose that $f$ is written as $f = \sum_{\kappa} a_\kappa x^\kappa.$ Then
the support of $f,$ denoted by $\mathrm{supp}(f),$ is defined as the set of those $\kappa  \in {\Bbb Z}_{\ge 0}^n$ such that $a_\kappa \ne 0.$ We denote the set $\Gamma(\mathrm{supp}(f))$ by $\Gamma(f).$ This set will be called the {\em Newton polyhedron at infinity} of $f.$ The polynomial $f$ is said to be {\em convenient} if $\Gamma(f)$ is convenient. If $f \equiv 0,$ then we set $\Gamma(f) = \emptyset.$ Note that, if $f$ is convenient, then for each nonempty subset $J$ of $\{1, \ldots, n\},$ we have $\Gamma(f) \cap {\Bbb R}^J = \Gamma(f|_{{\Bbb R}^J}).$ The {\em Newton boundary at infinity} of $f$, denoted by $\Gamma_{\infty}(f),$ is defined as the set of the faces of $\Gamma(f)$ which do not contain the origin $0$ in ${\Bbb R}^n.$

Let us fix a face $\Delta$ of $\Gamma_\infty(f).$  We define the {\em principal part of $f$ at infinity with respect to $\Delta,$} denoted by $f_\Delta,$ as the sum of the terms $a_\kappa x^\kappa$ such that $\kappa \in \Delta.$

\subsection{Non-degeneracy at infinity}

In \cite{Khovanskii1978,Kouchnirenko1976} (see also \cite{Dinh2013,Dinh2014,HaHV2013}), the authors introduced some conditions of non-degeneracy for polynomial maps in terms of Newton polyhedra. Moreover, some conditions of non-degeneracy for matrices were also given by \cite{Esterov2007}. We present here a new condition of non-degeneracy at infinity for symmetric polynomial matrices. This condition implies the condition in \cite{Khovanskii1978,Kouchnirenko1976} when the matrices considered are diagonal.

Let $F \colon \Bbb R^n \rightarrow {\mathcal S}^p,\ x\mapsto F(x)=(f_{ij}(x))_{i,j=1,\ldots,p},$ be a symmetric polynomial matrix. Let $\Gamma(F)$ denote the Minkowski sum $\sum_{i,j=1,\ldots,p}\Gamma(f_{ij})$ and we denote by $\Gamma_\infty(F)$ the set of faces of $\Gamma(F)$ which do not contain the origin $0$ in ${\Bbb R}^n.$ Let $\Delta$ be a face of the $\Gamma(F).$ According to Lemma~\ref{Lemma2}, we have the following decomposition $\Delta = \sum_{i,j=1,\ldots,p}\Delta_{ij}$ where $\Delta_{ij}$ is a face of $\Gamma(f_{ij}),$ for all $i,j = 1, \ldots, p.$ We denote by $F_\Delta$ the symmetric polynomial matrix $(f_{ij, \Delta_{ij}})_{i,j=1,\ldots,p}  \colon {\Bbb R}^n \rightarrow {\mathcal S}^p.$

\begin{definition}{\rm
We say that the polynomial matrix $F(x) = (f_{ij}(x))_{i,j=1,\ldots,p}$ is {\em non-degenerate at infinity} if and only if for any face $\Delta$ of $\Gamma_\infty(F)$ and for all $x \in ({\Bbb R} \setminus \{0\})^n,$ we have
\begin{equation*}
\left\{
\begin{array}{llll}
\Omega = (\omega_{ij})_{p \times p} \in {\mathcal S}^p, \omega_{ii} \ge 0 \textrm{ for } i = 1, \ldots, p, \ \tr(\Omega)=1, \\
\tr\left(\Omega\displaystyle\frac{\partial F_\Delta}{\partial x_k}(x)\right)=0 \ \text{ for }k=1,\ldots,n
\end{array}\right\}  \Rightarrow \tr(\Omega F_\Delta(x)) \ne 0.
\end{equation*}
}\end{definition}

\begin{remark}\label{Relax}{\rm
Note that the condition $\tr(\Omega)=1$ in the above definition can be replaced by $\tr(\Omega)\ne 0$.
}\end{remark}


\section{Nonsmooth \L ojasiewicz gradient inequality for the largest eigenvalue function} \label{LojasiewiczGradient}

In this section, we prove Theorem \ref{NonSmoothTheorem} which establishes a nonsmooth version of \L ojasiewicz gradient inequality with an explicit exponent for the largest eigenvalue function. 

Note that (\ref{Eq5b}) follows trivially from (\ref{Eq5}) since $\partial f(x)\subset\partial^\circ f(x)$, so it remains to prove (\ref{Eq5}). First of all, for each $x \in \mathbb{R}^n,$ we denote by $E(x)$ the set of unit eigenvectors of $F(x)$ corresponding to the eigenvalue $f(x),$ i.e.,
$$E(x) := \Big\{v \in \mathbb{S}^{p - 1} \ : \ F(x) v - f(x)v = 0 \Big\}.$$
Clearly, $E(x)$ is a compact set. Furthermore, we have the following stability result of the set of unit eigenvectors $E(x)$:

\begin{lemma} \label{Lemma3}
Let $\bar{x} \in \mathbb{R}^n.$ For each $\epsilon > 0$ there exists a constant $c > 0$ such that
$$E(x) \subset  E(\bar{x}) + c\|x - \bar{x}\|^{\frac{1}{\mathscr{R}(p, 4)}} \mathbb{B}^n \quad \textrm{ for all } \quad x \in \mathbb{B}^n(\bar{x}, \epsilon).$$
\end{lemma}
\begin{proof}
Consider the polynomial function
$$\Phi \colon \mathbb{R}^n \times \mathbb{R}^p \rightarrow \mathbb{R}, \quad (x, v) \mapsto \Phi(x, v) := \left ( \sum_{i = 1}^p v_i^2 - 1 \right )^2 + \left \| F(x) v - f(x) v \right\|^2 .$$
By definition, we have that $\Phi(x, v) \ge 0$ for all $(x, v) \in \mathbb{R}^n \times \mathbb{R}^{p}$ and that
\begin{eqnarray*}
E(x)
&=& \{v \in \mathbb{S}^{p - 1} \ : \ F(x) v - f(x)v = 0\} \\
&=& \{v \in \mathbb{R}^p \ : \ \sum_{i = 1}^p v_i^2 - 1 = 0 \ \textrm{ and } \ F(x) v - f(x)v = 0\} \\
&=& \{v \in \mathbb{R}^{p} \ : \ \Phi(x, v) =  0\}.
\end{eqnarray*}
Since the sphere $\mathbb{S}^{p -1}$ is a compact set, it follows from the \L ojasiewicz inequality (see, for example, \cite{Kurdyka2014,Pham2012}) that there is a constant $c > 0$ such that
\begin{eqnarray*}
c\, \mathrm{dist}(v , E(\bar{x})) & \le & [\Phi(\bar{x}, v)]^{\frac{1}{\mathscr{R}(p, 4)}} \quad \textrm{ for all } \quad v \in \mathbb{S}^{p -1}.
 \end{eqnarray*}

On the other hand, it is clear that the function $\Phi$ is locally Lipschitz, and so it is globally Lipschitz on the compact set $\mathbb{B}^n(\bar{x}, \epsilon) \times \mathbb{S}^{p -1}.$ Hence, there exists a constant $L > 0$ such that
 \begin{eqnarray*}
|\Phi(x, v) - \Phi(\bar{x}, v)|  & \le & L \|x - \bar{x}\| \quad \textrm{ for all } \quad (x, v) \in \mathbb{B}^n(\bar{x}, \epsilon) \times \mathbb{S}^{p -1}.
\end{eqnarray*}

Let $x \in \mathbb{B}^n(\bar{x}, \epsilon)$ and take an arbitrary  $v \in E(x).$ Then $\Phi(x, v) = 0,$ and therefore,
 \begin{eqnarray*}
c\, \mathrm{dist}(v , E(\bar{x}))
& \le & [\Phi(\bar{x}, v)]^{\frac{1}{\mathscr{R}(p, 4)}} \\
& = & |\Phi(x, v) - \Phi(\bar{x}, v)|^{\frac{1}{\mathscr{R}(p, 4)}} \\
& \le & L^{\frac{1}{\mathscr{R}(n, 4)}} \|x - \bar{x}\|^{\frac{1}{\mathscr{R}(p, 4)}}.
 \end{eqnarray*}
This implies immediately the lemma.
\end{proof}

For simplicity, we will write $g(x, v) := \langle F(x)v, v\rangle.$ For each integer $r \in \{1, \ldots, n + 1\},$ we define the function
$$G_r \colon \mathbb{R}^n \times \mathbb{R}^{r - 1} \times \mathbb{R}^{rp} \to \Bbb R, \quad (x, \lambda, v^1, \ldots, v^r) \mapsto
G_r(x, \lambda, v^1, \ldots, v^r),$$
by
$$G_r(x,\lambda, v^1, \ldots, v^r) :=  \displaystyle\sum_{l=1}^{r-1}\lambda_lg(x,v^l)+\left(1-\sum_{l=1}^{r-1}\lambda_l\right)g(x,v^r),$$
where $\lambda := (\lambda_1,\ldots,\lambda_{r-1}) \in \mathbb{R}^{r - 1}$ and $v^l :=(v_1^l,\ldots,v_p^l) \in \mathbb{R}^{p},\ l=1,\ldots,r.$ Clearly, $G$ is a polynomial of $n + r - 1 + rp$ variables with degree at most $d + 3.$ Define further the set ${\mathbf P} \subset\R^{r - 1}$ by
$$\mathbf{P} := \Big\{\lambda := (\lambda_1,\ldots,\lambda_{r-1})  \in\mathbb{R}^{r - 1} \ : \ \sum_{j=1}^{r - 1}\lambda_j \le 1 \ \textrm{ and } \
\lambda_j \ge 0, \textrm{ for } j = 1, \ldots, r - 1 \Big\}.$$

\begin{lemma}\label{Lemma4}
There exist some positive constants $c$ and $\epsilon$ such that
\begin{equation*}
\|\nabla G_r(x,\lambda,v^1,\ldots, v^r)\|\ge c\, |G_r(x, \lambda, v^1, \ldots, v^r)|^{\theta_r}
\end{equation*}
for all $x \in \mathbb{B}^n(\bar{x}, \epsilon)$, all $\lambda \in \mathbf{P},$ and all $v^l \in \mathbb{S}^{p - 1}$ with $\mathrm{dist}(v^l, E(\bar{x})) \le \epsilon$ for $l = 1, \ldots, r,$ where $\theta_r := 1 - \frac{1}{\mathscr{R}(n + r - 1 + rp,d+3)}$.
\end{lemma}
\begin{proof}
By a standard compactness argument, it suffices to show, for each $\bar{\lambda} \in \mathbf{P}$ and each $\bar{v}^1, \ldots, \bar{v}^r \in E(\bar{x}),$ that there exist some positive constants $\bar{c}$ and $\bar{\epsilon}$  such that
\begin{equation}\label{Eq6}
\|\nabla G_r(x,\lambda,v^1,\ldots, v^r)\|\ge \bar{c}\, | G_r(x, \lambda, v^1, \ldots, v^r)|^{\theta_r}
\end{equation}
for $\|x - \bar{x} \| \le \bar{\epsilon},$ $\|\lambda - \bar{\lambda}\| \le \bar{\epsilon},$ and $\|v ^l- \bar{v}^l \| \le \bar{\epsilon}$ for $l = 1, \ldots, r.$

To see this, take any $\bar{\lambda} \in \mathbb{P}$ and $\bar{v}^1, \ldots, \bar{v}^r \in E(\bar{x}).$  If $G_r(\bar{x}, \bar{\lambda}, \bar{v}^1, \ldots, \bar{v}^r) = 0$ then Inequality~(\ref{Eq6}) follows from Theorem~\ref{GradientInequality}. So we assume that $G_r(\bar{x}, \bar{\lambda}, \bar{v}^1, \ldots, \bar{v}^r) \ne 0.$ By definition, we have for all $l = 1, \ldots, r,$
\begin{eqnarray*}
f(\bar{x}) \bar{v}^l &=& F(\bar{x}) \bar{v}^l, \ \textrm{ and }  \\
f(\bar{x}) &=& \langle F(\bar{x}) \bar{v}^l, \bar{v}^l \rangle = g(\bar{x}, \bar{v}^l) =
G_r(\bar{x}, \bar{\lambda}, \bar{v}^1, \ldots, \bar{v}^r).
\end{eqnarray*}
Further, observe that
\begin{eqnarray*}
\nabla_{v^1, \ldots, v^r} G_r(\bar{x}, \bar{\lambda}, \bar{v}^1, \ldots, \bar{v}^r)
&=& \left[ 2 \bar{\lambda}_1 F( \bar{x} ) \bar{v}^1, \ldots, 2\bar{\lambda}_{r - 1}F(\bar{x}) \bar{v}^{r - 1}, 2\left(1 - \sum_{l = 1}^{r - 1}\bar{\lambda}_l\right)F(\bar{x})\bar{v}^r\right]\\
&=& \left[ 2 \bar{\lambda}_1 f( \bar{x} ) \bar{v}^1, \ldots, 2\bar{\lambda}_{r - 1}f(\bar{x}) \bar{v}^{r - 1}, 2\left(1 - \sum_{l = 1}^{r - 1}\bar{\lambda}_l\right) f(\bar{x})\bar{v}^r\right]\\
&=& \left[ 2 \bar{\lambda}_1 \bar{v}^1, \ldots, 2\bar{\lambda}_{r - 1}\bar{v}^{r - 1}, 2\left(1 - \sum_{l = 1}^{r - 1}\bar{\lambda}_l\right) \bar{v}^r\right]f( \bar{x} ),
\end{eqnarray*}
where $\nabla_{v^1, \ldots, v^r} G_r$ stands for the derivative of the function $G_r$ with respect to the variables $v^1, \ldots, v^r.$ Hence $ \nabla_{v^1, \ldots, v^r} G_r(\bar{x}, \bar{\lambda}, \bar{v}^1, \ldots, \bar{v}^r) \ne 0,$ and so $\nabla G_r(\bar{x}, \bar{\lambda}, \bar{v}^1, \ldots, \bar{v}^r) \ne 0$. Since $G_r$ and $\nabla G_r$ are continuous functions, by choosing $\bar c$ and $\bar\epsilon$ small enough, we get Inequality~(\ref{Eq6}).
\end{proof}

Now, we are in position to finish the proof of Theorem \ref{NonSmoothTheorem}.

\begin{proof}[Proof of Theorem \ref{NonSmoothTheorem}]
Without loss of generality we may assume that $f(\bar x) = 0.$

Applying Lemma \ref{Lemma3} for $\epsilon_1 := 1$ we get a constant $c > 0$ such that
$$E(x) \subset  E(\bar{x}) + c\|x - \bar{x}\|^{\frac{1}{\mathscr{R}(p, 4)}} \mathbb{B}^n \quad \textrm{ for all } \quad x \in \mathbb{B}^n(\bar{x}, \epsilon_1).$$
Let $\epsilon_2 > 0$ be such that Lemma \ref{Lemma4} holds and choose a real number $\epsilon$ satisfying
$0 < \epsilon \le \min\{\epsilon_1, \epsilon_2, \left(\frac{\epsilon_2}{c} \right)^{\mathscr{R}(p, 4)}\}.$ Then it is clear that
$$\mathrm{dist}(v, E(\bar{x})) \le \epsilon \quad \textrm{ for all } \quad
x \in \mathbb{B}^n(\bar{x}, \epsilon) \ \textrm{ and } \ v \in E(x).$$
Shrinking $\epsilon$, if necessary, we may assume that $|f(x)| < 1$ for all $x \in \mathbb{B}^n(\bar{x}, \epsilon).$

Take an arbitrary $x$ in $\mathbb{B}^n(\bar{x}, \epsilon)$ and let $w \in \partial^\circ f(x).$ By Lemma \ref{Lemma1}, there are $(\lambda_1, \ldots, \lambda_{r - 1}) \in \mathbf{P}$ and some unit eigenvectors $v^1, \ldots, v^r$ of $F(x)$ corresponding to the eigenvalue $f(x)$,
such that
$$w = \displaystyle\sum_{l=1}^{r-1}\lambda_l\nabla_xg(x,v^l)+\left(1-\sum_{l=1}^{r-1}\lambda_l\right)\nabla_xg(x,v^r),$$
for some $r \in \{1, \ldots, n + 1\}.$ Since $g(x,v^l) = f(x)$ for $l = 1, \ldots, r,$ it follows that
\begin{eqnarray*}
G_r(x,\lambda,v^1,\ldots,v^r) &=& \displaystyle\sum_{l=1}^{r-1}\lambda_lf(x)+\left(1-\sum_{l=1}^{r-1}\lambda_l\right)f(x)=f(x).
\end{eqnarray*}
Moreover we have
\begin{displaymath}
\begin{array}{llll}
$$\nabla G_r(x,\lambda,v^1,\ldots,v^r)
&=& \left[\displaystyle\sum_{l=1}^{r-1}\lambda_l\nabla_xg(x,v^l)+\left(1-\sum_{l=1}^{r-1}\lambda_l\right)\nabla_xg(x,v^r),\right.$$\\
$$&& \quad \displaystyle g(x,v^1)-g(x,v^r),\ldots,g(x,v^{r-1})-g(x,v^r),$$\\
$$&& \quad \left.\displaystyle 2\lambda_1F(x)v^1,\ldots,2\lambda_{r-1}F(x)v^{r-1},2\left(1-\sum_{l=1}^{r-1}\lambda_l\right)F(x)v^r\right]$$\\
$$&=& \Big [w, f(x)-f(x),\ldots,f(x)-f(x),$$\\
$$&&\quad  \left.\displaystyle 2\lambda_1f(x)v^1,\ldots,2\lambda_{r-1}f(x)v^{r-1},2\left(1-\sum_{l=1}^{r-1}\lambda_l\right) f(x)v^r\right]$$\\
$$&=&\displaystyle\left[w,0,\ldots,0,2\lambda_1f(x)v^1,\ldots,2\lambda_{r-1}f(x)v^{r-1},2\left(1 - \sum_{l=1}^{r-1}\lambda_l\right)f(x)v^r\right].$$
\end{array}
\end{displaymath}
Therefore
\begin{eqnarray*}
\|\nabla G_r(x,\lambda,v^1,\ldots,v^r)\| &=& \|w\|+2|f(x)|\left(\displaystyle\sum_{l=1}^{r-1}\lambda_l\|v^l\| + \left(1 - \sum_{l=1}^{r-1}\lambda_l \right)\|v^r\|\right) \\
&=&\|w\|+2|f(x)|.
\end{eqnarray*}
(Here we use the norm $\|(x, \lambda, v^1, \ldots, v^r)\| := \|x\| + \|\lambda\| + \|v^1\|+ \cdots + \|v^r\|.)$
For each $l = 1, \ldots, r,$ we know that $v^l \in E(x)$ and so
\begin{eqnarray*}
\mathrm{dist}(v^l, E(\bar{x})) & \le & \epsilon \ \le \ \epsilon_2 \quad \textrm{ for } \quad l = 1, \ldots, r.
\end{eqnarray*}
Thanks to Lemma \ref{Lemma4}, we have
\begin{eqnarray*}
\|w\|+2|f(x)| \ = \ \|\nabla G_r(x,\lambda,v^1,\ldots,v^r)\| &\ge& c|f(x)|^{\theta_{r}} \ \ge \ c |f(x)|^{\theta_{n+1}},
\end{eqnarray*}
where the last inequality follows from the facts that $|f(x)| <1$ and
$$0 < \theta_r < \theta_{n + 1} = 1 - \frac{1}{\mathscr{R}(2n + p(n + 1), d + 3)}.$$
Thus
$$\|w\|\ge (c - 2|f(x)|^{1 - \theta_{n+1}})|f(x)|^{\theta_{n+1}}.$$
By choosing $\epsilon$ small enough, then
$$\|w\|\ge \frac{c}{2}|f(x)|^{\theta_{n+1}} \quad \text{ for all } \quad x \in \mathbb{B}^n(\bar{x}, \epsilon).$$
The inequality holds for all $w \in \partial^\circ f(x),$ so the theorem follows.
\end{proof}

\section{Local \L ojasiewicz-type inequality and local separation of semialgebraic sets} \label{LocalErrorBounds}

Theorem \ref{NonSmoothTheorem} allows us to deduce the following local \L ojasiewicz-type inequality for the largest eigenvalue function.

\begin{theorem}\label{CompactErrorBound}
Let $F$ and $f$ be as in Theorem \ref{NonSmoothTheorem}. Then for any compact set $K \subset\Bbb R^n,$ there exists a constant $c > 0$ such that
\begin{equation}\label{Eq7}
c\, \mathrm{dist}(x, S_F) \le [f(x)]_+^{\frac{1}{\mathscr{R}(2n+p(n+1), d+3)}} \quad \textrm{ for all } \quad x\in K,
\end{equation}
where $S_F := \{x \in {\Bbb R}^n \ : \ F(x) \preceq  0\}$.
\end{theorem}

\begin{proof}
Theorem~\ref{CompactErrorBound} can be deduced straightly from a result on local error bounds (see \cite{Ngai2009}). Here, we present a different proof whose ideas is based on estimating the length of trajectories of the subgradient dynamical system (see e.g. \cite{Bolte2007}, \cite{Kurdyka2014} for more details). 

Denote by $\overset{\circ}{\mathbb B^n}(x,\epsilon)$ the open ball centered at $x$ of radius $\epsilon > 0$ in $\mathbb{R}^n$ (and so ${\mathbb{B}^n}(x,\epsilon)$ is its closure). Since $K$ is compact, we can cover $K$ by finite open balls ${\stackrel{\circ}{\mathbb{B}^n}}(\bar x_i,\epsilon_i), i=1,\ldots,m,$ such that:
\begin{itemize}
\item Either $\bar x_i\in S_F$ or $\mathbb{B}^n(\bar x_i,\epsilon_i)\cap S_F=\emptyset$;
\item If $\bar x_i\in S_F$ then Inequality~(\ref{Eq5b}) holds in ${\stackrel{\circ}{\mathbb{B}^n}}(\bar x_i,\epsilon_i+\tilde cM)$ where $\tilde c :=\frac{\mathscr{R}(2n+p(n+1),d+3)}{c},$ $M := \displaystyle\sup_{x \in \mathbb{B}^n(\bar x_i,\epsilon_i)}[f_+(x)]^{\frac{1}{\mathscr{R}(2n+p(n+1),d+3)}},$ and $c$ is the constant in Inequality (\ref{Eq5b}).
\end{itemize}

First of all, it is clear that by taking $c$ small enough, Inequality~(\ref{Eq7}) holds for all $x\in {\stackrel{\circ}{\mathbb{B}^n}}(\bar x_i,\epsilon_i)$ with $\mathbb{B}^n(\bar x_i,\epsilon_i)\cap S_F=\emptyset$ since $\inf_{x\in \mathbb{B}^n(\bar x_i,\epsilon_i)}[f(x)]_+=\inf_{x \in \mathbb{B}^n(\bar x_i,\epsilon_i)}[f(x)] > 0.$ So it remains to prove Inequality~(\ref{Eq7}) for all $x\in {\stackrel{\circ}{\mathbb{B}^n}}(\bar x_i,\epsilon_i)$ with $\bar x_i\in S_F.$ 

Note that $f_+(x) = \max_{\{\|v\| = 1\} \cup \{0\}} \langle F(x) v, v \rangle \ge 0$ for all $x \in \mathbb{R}^n.$ Since the set $\{\|v\| = 1\} \cup \{0\}$ is nonempty compact, $f_+$ is locally Lipschitz and locally representable as a difference of a convex continuous and a convex quadratic function (see, e.g., \cite[Theorem~10.33]{Rockafellar1998}). In particular, it satisfies $\partial f_+(x) = \hat{\partial} f_+(x) \ne \emptyset$ for all $x \in \mathbb{R}^n.$ Furthermore, if $f(x) > 0$ then $f(x) = f_+(x)$ and $\frak m_{f_+}(x) = \frak m_{f}(x)$ by the continuity of $f.$

Let $x \in {\stackrel{\circ}{\mathbb{B}^n}}(\bar x_i,\epsilon_i)$ be such that $f(x)>0.$ By \cite[Corollaries~4.1 and 4.2]{Bolte2007}, there is a unique absolutely continuous integral curve $u \colon [0, +\infty) \to \Bbb R^n$ of the dynamical system 
$$0 \in \dot{u}(t) + \partial[f_+(u(t))] \quad \textrm{ with } \quad u(0) = x$$
such that the following properties hold:
\begin{itemize}
\item [(a)] The function $f_+ \circ u$ is absolutely continuous and decreasing on $[0, +\infty).$

\item [(b)] For almost all $t \in (0, +\infty),$
\begin{eqnarray*}
\| \dot{u}(t)\| &=& \frak m_{f_+}(u(t)) \quad \textrm{ and } \quad  \frac{d}{dt}({f_+} \circ u)(t) \ = \ -[\frak m_{f_+}(u(t))]^2. 
\end{eqnarray*}

\item [(c)] If there exists $t_0 > 0$ such that $\frak{m}_{f_+} (u(t_0)) = 0,$ then $u(t) = u(t_0)$ for all $t \ge  t_0.$
\end{itemize}

Let  
$$T := \sup\{t > 0 : f_+(u(t))  > 0\}.$$
Clearly, $0 < T \le +\infty.$ Assume $f_+(u(t_0))  = 0$ for some $t_0 \in (0, T).$ Then $u(t_0)$ is a global minimizer of $f_+$ on $\mathbb{R}^n.$ Hence $0 \in \partial f_+(u(t_0))$ and in consequence $\frak{m}_{f_+} (u(t_0)) = 0.$ The property~(c) implies that for all $t \ge  t_0,$ $u(t) = u(t_0)$ and so $f_+(u(t))  = f_+(u(t_0))  = 0,$  which contradicts the definition of $T.$ Therefore $f_+(u(t))  > 0$ for all $t \in [0, T).$ 

Let
$$T_0 := \sup\{t \in (0, T)  :  u(s)\in {\stackrel{\circ}{\mathbb{B}^n}}(\bar x_i, \epsilon_i + \tilde cM) \quad \textrm{ for all } \quad s \in [0, t)\}.$$
Clearly, $0 < T_0 \le T.$ For simplicity we write $\rho := 1 - \frac{1}{\mathscr{R}(2n+p(n+1),d+3)}	> 0.$ By chain rule, the property~(b) and Theorem~\ref{NonSmoothTheorem}, we have for almost all $t \in  [0, T_0),$
\begin{eqnarray*}
\frac{d}{dt}({f_+}\circ u)^{1 - \rho} (t) 
&=& (1 - \rho) ({f_+}\circ u)^{-\rho}(t) \frac{d}{dt}({f_+}\circ u) (t)  \\
&=& -(1 - \rho) ({f_+}\circ u)^{-\rho}(t) [\frak{m}_{f_+}(u(t))]^2 \\
&=& -(1 - \rho) ({f_+}\circ u)^{-\rho}(t)[\frak{m}_{f_+}(u(t))]\|\dot{u}(t)\| \\
&=& -(1 - \rho) ({f}\circ u)^{-\rho}(t) [\frak{m}_{f}(u(t))]\|\dot{u}(t)\| \\
&\le& -\frac{1}{\tilde{c}} \|\dot{u}(t)\|.
\end{eqnarray*}
In consequence, we obtain for all $t \in  [0, T_0),$
\begin{eqnarray} \label{LengthEqn}
\mathrm{length}(u|_{[0, t)})
&\le&  \tilde{c} \left[ ({f_+}\circ u)^{1 - \rho} (0)  - ({f_+}\circ u)^{1 - \rho} (t) \right], 
\end{eqnarray}
where $\mathrm{length}(u|_{[0, t)})$ stands for the length of $u|_{[0, t)}.$

Assume that we have proved that $\lim_{t \to T_0} f_+(u(t)) = 0.$ This, of course, implies that 
\begin{eqnarray*}
\mathrm{dist}(x, S_F) & \le & \mathrm{length}(u|_{[0, T_0)}) \ \le \ {\tilde c}[f_+(x)]^{1 - \rho},
\end{eqnarray*}
which completes the proof of the theorem.

So we are left with proving that $\lim_{t \to T_0} f_+(u(t)) = 0.$ Indeed, by contradiction, assume that $\lim_{t \to T_0} f_+(u(t)) > 0.$ Then, by \eqref{LengthEqn},
\begin{eqnarray*}
\mathrm{length}(u|_{[0, T_0)}) &  <   & \tilde{c} [f_+(x)]^{1 - \rho} \le \tilde{c} M.
\end{eqnarray*}

There are two cases to be considered.

\subsubsection*{Case 1: $T_0 < + \infty$}
In this case, we have
\begin{eqnarray*}
\|u(T_0) - \bar{x}_i \| 
&  \le  & \mathrm{length}(u|_{[0, T_0)}) +  \|x - \bar{x}_i \|  \  <   \ {\tilde c} M + \epsilon_i,
\end{eqnarray*}
which yields $u(T_0) \in {\stackrel{\circ}{\mathbb{B}^n}}(\bar x_i,\epsilon_i+\tilde cM).$ Then, by continuity, $f_+(u(T_0 + \delta)) > 0$ and  $u(T_0 + \delta) \in {\stackrel{\circ}{\mathbb{B}^n}}(\bar x_i,\epsilon_i+\tilde cM)$ 
for all sufficiently small $\delta > 0.$ This contradicts the definition of $T_0.$ 

\subsubsection*{Case 2: $T_0 = + \infty$}
In this situation, the trajectory $u(t)$ is bounded because its length is bounded by ${\tilde c} M.$ Thanks to \cite[Theorem~4.5]{Bolte2007}, the trajectory $u(t)$ converges to some point $x^\infty \in \mathbb{R}^n$ with  $\frak{m}_{f_+}(x^\infty) = 0.$ Arguing as above, it is easy to see that $x^\infty \in {\stackrel{\circ}{\mathbb{B}^n}}(\bar x_i,\epsilon_i+\tilde cM).$  This, together with Theorem~\ref{NonSmoothTheorem}, implies that $f_+(x^\infty) = 0,$ which contradicts our assumption that $\lim_{t \to T_0} f_+(u(t)) > 0.$ 
\end{proof}

Another consequence of Theorem \ref{NonSmoothTheorem} is the following separation of semialgebraic sets with an explicit exponent.

\begin{proposition}\label{Separation}
Let $F \colon  \Bbb R^n \rightarrow {\mathcal S}^p,\ x\mapsto F(x)=(f_{ij}(x)),$ and $G \colon \Bbb R^n \rightarrow {\mathcal S}^q,\ x\mapsto G(x)=(g_{kl}(x)),$ be two symmetric polynomial matrices of order $p$ and $q,$ respectively. Set
$$S_F := \{x \in {\Bbb R}^n \ : \ F(x) \preceq  0\} \quad \textrm{ and } \quad   S_G := \{x \in {\Bbb R}^n \ : \ G(x) \preceq  0\},$$
and assume that $S_F\cap S_G \ne\emptyset.$ Then for any compact set $K \subset\Bbb R^n,$ there exists a constant $c>0$ such that
$$c\, \mathrm{dist}(x,S_F\cap S_G) \le \Big (\mathrm{dist}(x,S_F)+\mathrm{dist}(x,S_G) \Big)^{\frac{1}{\mathscr{R}(2n+(p+q)(n+1), d+3)}} \quad \textrm{ for all } \quad x \in K,$$
where $\displaystyle d := \max_{i,j = 1, \ldots, p,\ k,l = 1, \ldots, q}\{\deg f_{i j},\deg g_{k l}\}.$
\end{proposition}
\begin{proof}
Let $f(x)$ and $g(x)$ be the largest eigenvalues of the matrices $F(x)$ and $G(x),$ respectively. It follows from Lemma
\ref{Lemma1} that
$$f(x) := \max_{\|v\|=1}\langle F(x)v,v\rangle \quad \textrm{ and } \quad  g(x) := \max_{\|u\|=1}\langle G(x)u,u\rangle.$$
Define the symmetric polynomial matrix $H \colon  \Bbb R^n \rightarrow {\mathcal S}^{p + q},\ x\mapsto H(x),$ as follows
$$H(x) :=
\begin{pmatrix}
F(x) & 0 \\
0 & G(x)
\end{pmatrix},$$
and set $S_H := \{x \in {\Bbb R}^n \ : \ H(x) \preceq  0\}.$ It is clear that $S_H = S_F\cap S_G.$

Let $h(x)$ be the largest eigenvalue of the matrix $H(x).$ It is clear that $h(x)= \max\{f(x),g(x)\}$, so
\begin{displaymath}
\begin{array}{lrl}
$$[h(x)]_+&=&[\max\{f(x),g(x)\}]_+$$\\
$$&=&\max\{f(x),g(x),0\}$$\\
$$&=& \max\{[f(x)]_+,[g(x)]_+\}$$\\
$$&\le& [f(x)]_++[g(x)]_+.$$
\end{array}
\end{displaymath}
By Theorem \ref{CompactErrorBound}, there exists a constant $c>0$ such that for all $x\in K$, we have
$$c\, \mathrm{dist}(x, S_H) \le [h(x)]_+^{\frac{1}{\mathscr{R}(2n+(p+q)(n+1), d+3)}}.$$
Therefore
\begin{eqnarray}\label{Eq8}
c\, \mathrm{dist}(x, S_F\cap S_G) &\le& ([f(x)]_++[g(x)]_+)^{\frac{1}{\mathscr{R}(2n+(p+q)(n+1), d+3)}}.
\end{eqnarray}
Since $K$ is compact, $M := \max_{x\in K}\{\mathrm{dist}(x,S_F),\mathrm{dist}(x,S_G)\}<+\infty$ and $\widetilde{K} := K + M\mathbb{B}^n$ is a compact set. Note that the functions $x \mapsto [f(x)]_+$ and $x \mapsto [g(x)]_+$ are locally Lipschitz, so are globally Lipschitz on the compact set $\widetilde{K}$. Thus there exists a constant $L>0$ such that for all $x,y\in \widetilde{K}$, we have
$$|[f(x)]_+-[f(y)]_+|\le L\|x-y\| \quad \textrm{ and } \quad |[g(x)]_+-[g(y)]_+|\le L\|x-y\|.$$
Now for each $x\in K$, there exist $y\in S_F$ and $z\in S_G$ such that
$$\mathrm{dist}(x,S_F) = \|x-y\| \quad \textrm{ and } \quad \mathrm{dist}(x, S_G) = \|x-z\|.$$
It is clear that $y, z \in \widetilde{K}.$ Hence
\begin{eqnarray*}
|[f(x)]_+| &=&  |[f(x)]_+-[f(y)]_+|\le L\|x-y\|=L\mathrm{dist}(x,S_F), \\
|[g(x)]_+| &=& |[g(x)]_+-[g(z)]_+|\le L\|x-z\|=L\mathrm{dist}(x,S_G).
\end{eqnarray*}
These inequalities, together with Inequality (\ref{Eq8}), imply the proposition.
\end{proof}

The next result establishes a sharpen version of \L ojasiewicz's factorization lemma.
\begin{corollary}\label{factorization}
Let $F \colon  \Bbb R^n \rightarrow {\mathcal S}^p,\ x\mapsto F(x)=(f_{ij}(x)),$ $G \colon \Bbb R^n \rightarrow {\mathcal S}^q,\ x\mapsto G(x)=(g_{kl}(x)),$ and $H \colon \Bbb R^n \rightarrow {\mathcal S}^r,\ x\mapsto H(x)=(h_{st}(x)),$ be some symmetric polynomial matrices of order $p, q,$ and $r,$ respectively. Let $f(x), g(x),$ and $h(x)$ be the corresponding largest eigenvalue functions of $F(x), G(x),$ and $H(x)$. Assume that $K := \{x \in {\Bbb R}^n \ : \ H(x) \preceq  0\}$ is a compact set and that
\begin{eqnarray*}
\{x \in K \ : \ f(x) \le 0 \} & \subset & \{x \in K \ : \  g(x) \le 0\}.
\end{eqnarray*}
Then there is a constant $c > 0$ such that
$$[g(x)]_+ \le c\,[f(x)]_+^{\frac{1}{\mathscr{R}(2n + (p + r)(n + 1), d + 3)}}, \quad \textrm{for all } \quad x\in K,$$
where $\displaystyle d := \max_{i,j = 1, \ldots, p,\ s, t = 1, \ldots, r}\{\deg f_{i j},\deg h_{s t}\}.$ 
\end{corollary}
\begin{proof}
Let ${\mathcal A} := \{x \in K \ : \ f(x) \le 0\}.$ We have
\begin{eqnarray*}
{\mathcal A} &=& \left\{x \in \mathbb{R}^n \ : \ \begin{pmatrix} F(x) & 0 \\ 0 & H(x) \end{pmatrix} \preceq 0\right\}
\ = \ \{x \in \mathbb{R}^n \ : \ \max\{f(x), h(x)\} \le 0\} \\
&\subset& \{x \in K \ : \  g(x) \le 0\}.
\end{eqnarray*}
Since the set $K$ is compact, Theorem \ref{CompactErrorBound} gives 
\begin{eqnarray*}
\mathrm{dist}(x, {\mathcal A}) & \le & c_0\, \max\{f(x), h(x), 0\}^{\frac{1}{\mathscr{R}(2n + (p + r)(n + 1), d + 3)}} 
\  = \ c_0\, [f(x)]_+^{\frac{1}{\mathscr{R}(2n + (p + r)(n + 1), d + 3)}},
\end{eqnarray*}
for all $x \in K,$ where $c_0$ is a positive constant. Let $M:=\max_{x\in K}\mathrm{dist}(x,\{g\le 0\})<+\infty$ and $\tilde K:=K+M\Bbb B^n.$ The function $g$ is locally Lipschitz, thus, is globally Lipschitz on $\tilde K$, i.e., there is a constant $L > 0$ such that $|g(x) - g(y)| \le L \|x - y\|$ for all $x, y \in \tilde K.$

Now take any $x \in K.$ Clearly, there exists a point $y \in \tilde K$ such that $g(y) \le 0$ and $\mathrm{dist}(x, \{g \le 0\}) = \|x - y\|.$ Therefore,
\begin{eqnarray*}
[g(x)]_+ &\le& |g(x) - g(y)| \ \le \ L\|x - y\| \ = \  L d\big(x, \{g \le 0\} \big) \\
&\le& L d\big(x, {\mathcal A} \big) \ \le \ L c_0 [f(x)]_+^{\frac{1}{\mathscr{R}(2n + (p + r)(n + 1), d + 3)}}.
\end{eqnarray*}
This completes the proof of the corollary.
\end{proof}
\begin{remark}{\rm The statement of Corollary \ref{factorization} still holds in the case $g \colon K \rightarrow \mathbb{R}$ is a locally Lipschitz function.}
\end{remark}

\section{Global separation of semialgebraic sets and global \L ojasiewicz-type inequality} \label{GlobalErrorBounds}

In this section we provide some versions of global separation of semialgebraic sets and global \L ojasiewicz-type inequality with explicit exponents for the largest eigenvalue function.

\begin{corollary}\label{GlobalSeparation}
Let $F \colon  \Bbb R^n \rightarrow {\mathcal S}^p,\ x\mapsto F(x)=(f_{ij}(x)),$ and $G \colon \Bbb R^n \rightarrow {\mathcal S}^q,\ x\mapsto G(x)=(g_{kl}(x)),$ be two symmetric polynomial matrices of order $p$ and $q,$ respectively. Set
$$S_F := \{x \in {\Bbb R}^n \ : \ F(x) \preceq  0\} \quad \textrm{ and } \quad   S_G := \{x \in {\Bbb R}^n \ : \ G(x) \preceq  0\}$$
and assume that $S_F\cap S_G \ne\emptyset.$ Then there exists a constant $c>0$ such that
\begin{eqnarray*}
c \left( \frac{\mathrm{dist}(x,S_F\cap S_G)}{1  + \|x\|^2} \right)^{\mathscr{R}(2n+(p+q)(n+1), d+3)}  & \le &
\mathrm{dist}(x,S_F) + \mathrm{dist}(x,S_G) \quad \textrm{ for all } \quad x \in \mathbb{R}^n,
\end{eqnarray*}
where $\displaystyle d := \max_{i,j = 1, \ldots, p,\ k,l = 1, \ldots, q}\{\deg f_{i j},\deg g_{k l}\}.$
\end{corollary}
\begin{proof}
The proof follows the same lines of that of \cite[Theorem 2]{Kurdyka2014}, by using Proposition~\ref{Separation} instead of~\cite[Corollary 8]{Kurdyka2014}. Note that the arguments of the proof of~\cite[Theorem 2]{Kurdyka2014} also hold for semialgebraic sets, the assumption of algebraicity is only needed for the application of~\cite[Corollary 8]{Kurdyka2014}. We omit the details.
\end{proof}

\begin{remark} 
{\rm Corollary~\ref{GlobalSeparation} can be also obtained by applying \cite[Theorem~1.1]{Kurdyka2015} but the exponent will be different.
}\end{remark}

Next we state a global \L ojasiewicz-type inequality for the largest eigenvalue function (compare \cite[Theorem~7]{Solerno1991}):

\begin{corollary}\label{GlobalKollar}
Let $F \colon  \Bbb R^n \rightarrow {\mathcal S}^p,\ x\mapsto F(x)=(f_{ij}(x)),$ be a symmetric polynomial matrix of order $p,$ and assume that $S_F := \{x \in {\Bbb R}^n \ : \ F(x) \preceq  0\} \ne \emptyset.$
Then for some constant $c > 0,$ we have
\begin{eqnarray*}
c \left ( \frac{\mathrm{dist}(x, S_F)}{1 + \|x\|^2} \right)^{\mathscr{R}(2(n + 1) + (p + 2)(n + 2), d + 3)} & \le & [f(x)]_+ \quad \textrm{ for all } \quad x \in \mathbb{R}^n,
\end{eqnarray*}
where $\displaystyle d := \max_{i,j = 1, \ldots, p} \deg f_{i j}.$
\end{corollary}
\begin{proof}
Define symmetric polynomial matrices $\widetilde{F} \colon  {\Bbb R^n} \times \mathbb{R} \rightarrow {\mathcal S}^{p}$ and
$\widetilde{G} \colon  {\Bbb R^n} \times \mathbb{R} \rightarrow {\mathcal S}^{2}$ by
\begin{eqnarray*}
\widetilde{F} (x, y) &:=& F(x) - yI_p \quad \textrm{ and } \quad
\widetilde{G} (x, y) \ := \
\begin{pmatrix}
y & 0 \\
0 & - y
\end{pmatrix}
\end{eqnarray*}
for $x \in \mathbb{R}^n$ and $y \in \mathbb{R},$ where $I_p$ denotes the unit matrix of order $p.$ Let
$S_{\widetilde{F}} := \{(x, y) \in {\Bbb R}^n \times {\Bbb R}\ : \ \widetilde{F}(x, y) \preceq  0\}$ and
$S_{\widetilde{G}} := \{(x, y) \in {\Bbb R}^n \times {\Bbb R} \ : \ \widetilde{G}(x, y) \preceq  0\}.$ By
Corollary~\ref{GlobalSeparation}, there exists a constant $c > 0$ such that
\begin{eqnarray*}
c\left(
\frac{\mathrm{dist}(z, S_{\widetilde{F}} \cap S_{\widetilde{G}})}{1 + \|z\|^2} \right)^{\mathscr{R}(2(n + 1) + (p + 2)(n + 2), d + 3)} &\le&
\mathrm{dist}(z, S_{\widetilde{F}}) + \mathrm{dist}(z, S_{\widetilde{G}})
\end{eqnarray*}
for all $z := (x, y) \in \mathbb{R}^n \times \mathbb{R}.$ Now it is sufficient to consider $x \in \mathbb{R}^n$ satisfying $f(x) \ge 0.$ Clearly, $S_{\widetilde{G}} = \mathbb{R}^n \times \{0\},$ so $S_{\widetilde{F}} \cap S_{\widetilde{G}}=S_F\times\{0\}$ and $\mathrm{dist}((x,0), S_{\widetilde{G}})=0$. Moreover $\mathrm{dist}((x,0), S_{\widetilde{F}} \cap S_{\widetilde{G}})= \mathrm{dist}(x,S_F)$. Note that $(x, f(x)) \in S_{\widetilde{F}}.$ Thus
\begin{eqnarray*}
\mathrm{dist}((x,0), S_{\widetilde{F}}) & \le & \|(x, 0) - (x, f(x))\| \ = \ f(x).
\end{eqnarray*}
The corollary follows.
\end{proof}

As a direct consequence of Corollary \ref{GlobalKollar}, we obtain the following result.
(see \cite{Kollar1988, Kurdyka2014}):

\begin{corollary}
Let $F \colon  \Bbb R^n \rightarrow {\mathcal S}^p,\ x\mapsto F(x)=(f_{ij}(x)),$ be a symmetric polynomial matrix of order $p,$ and assume that $S_F := \{x \in {\Bbb R}^n \ : \ F(x) \preceq  0\}$ is a nonempty compact set. 
Then there are some constants $c > 0$ and $R > 0$ such that
\begin{eqnarray*}
c\|x\|^{-\mathscr{R}(2(n + 1) + (p + 2)(n + 2), d + 3)} & \le & [f(x)]_+,  \quad \textrm{ for all } \quad \|x\| \ge R,
\end{eqnarray*}
where $d := \max_{i, j = 1, \ldots, p} \deg f_{ij}.$
\end{corollary}
\begin{proof}
Indeed, since the set $S_F$ is compact, we can find some positive constants $c_1$ and $c_2$ satisfying the following inequality
$$c_1 \|x\| \le \mathrm{dist}(x, S_F) \le c_2 \|x\| \quad \textrm{ for } \quad \|x\| \gg 1.$$
This, combining with Corollary \ref{GlobalKollar}, yields the required conclusion.
\end{proof}

\section{Global \L ojasiewicz-type inequality and non-degeneracy at infinity} \label{GlobalHolderErrorBounds}

In this part, we prove Theorem \ref{HolderTypeTheorem} which establishes a global \L ojasiewicz-type inequality with an explicit exponent for the largest eigenvalue function of a symmetric polynomial matrix, which is non-degenerate at infinity.

The following lemma is a key to prove Theorem \ref{HolderTypeTheorem}.

\begin{lemma} \label{Lemma5}
Under the assumptions of Theorem \ref{HolderTypeTheorem}, there exist some constants $c > 0$ and $R > 0$ such that
$${\frak m}_f(x) \ge c \quad \textrm{ for all } \quad \|x\| \ge R.$$
\end{lemma}
\begin{proof}
By contradiction, assume that there exists a sequence $\{x^k\}_{k \in {\Bbb N}} \subset {\Bbb R}^n$ such that
$$\lim_{k \to \infty} \|x^k \|= \infty \quad \quad \textrm{ and } \quad \lim_{k \to \infty} {\frak m}_f(x^k) = 0.$$

By Lemma \ref{Lemma1}, for each $k$ there exist some nonnegative real numbers $\lambda_1^k,\ldots,\lambda_r^k,$ with $\sum_{l=1}^r\lambda_l^k=1,$ and $r$ unit eigenvectors $(v^1)^k,\ldots,(v^r)^k$ corresponding to $f(x)$ such that
$$\frak m_f(x^k)=\left \|\sum_{l=1}^r\lambda_l^k\nabla_x \left \langle F(x^k)(v^l)^k,(v^l)^k \right\rangle \right \|=\left\|\sum_{l=1}^r\lambda_l^k\sum_{i,j=1}^p(v_i^l)^k(v_j^l)^k\nabla f_{ij}(x^k)\right\|.$$
Note that $r = r(k) \le n + 1.$ By taking subsequence if necessary, we may suppose that $r$ does not depend on $k$.
Since the function $x \mapsto {\frak m}_f(x)$ is semi-algebraic, by Lemma \ref{Lemma1} and by applying Curve Selection Lemma at infinity (Lemma \ref{CurveSelectionLemma}) with the following setup: the set
\begin{equation*}
\begin{array}{lll}
A := \big \{(x,\lambda,v^1,\ldots,v^r)\in\Bbb R^n\times\Bbb R^r\times\Bbb R^{r\times p} :
	&\lambda=(\lambda_1,\ldots,\lambda_r),\ v^l=(v_1^l,\ldots,v_p^l),\ l=1,\ldots,r, \\
	&\lambda_l \ge  0, \ \sum_{i=1}^r \lambda_l=1, \\
	&\|v^l\|=1, \ F(x)v^l=f(x)v^l \big \}
\end{array}
\end{equation*}
which is a semi-algebraic set, the sequence $(x^k,\lambda^k,(v^1)^k,\ldots,(v^r)^k)\in A$ which tends to infinity as $k\to\infty,$ and the semi-algebraic function $x\mapsto {\frak m}_f(x),$
it follows that there exist a smooth semi-algebraic curve $\varphi(t) := (\varphi_1(t), \ldots, \varphi_n(t))$ and some smooth semi-algebraic functions $\lambda_l(t),\ v_i^l(t),\ l=1,\ldots,r,\ i=1,\ldots,p$, for $0 < t \ll 1,$ such that
\begin{enumerate}
  \item [(a)] $\lim_{t \to 0} \| \varphi(t)\| = \infty;$
  \item [(b)] $\lambda_l(t) \ge 0$ for all $l=1,\ldots,r,$ and $\sum_{l=1}^r\lambda_l(t) = 1;$
  \item [(c)] $\|v^l(t)\|=\|(v_1^l(t),\ldots,v_p^l(t))\|=1$ and $F(\varphi(t))v^l(t)=f(\varphi(t))v^l(t)$
  for all $l=1,\ldots,r;$
  \item [(d)] ${\frak m}_f(\varphi(t)) = \|\sum_{l=1}^r\lambda_l(t)\nabla_x\langle F(\varphi(t))v^l(t),v^l(t)\rangle\| \to 0$ as $t \to 0.$
\end{enumerate}

Let $I := \{s: \ \varphi_s(t)\not\equiv 0\}.$ By Condition (a),
$I \ne \emptyset.$ By Growth Dichotomy Lemma (Lemma~\ref{GrowthDichotomyLemma}), for $s \in I,$ we can expand the coordinate function $\varphi_s$ in terms of the parameter $t$ as follows
$$\varphi_s(t) =  x_s^0 t^{q_s} + \textrm{ higher order terms in } t,$$
where $x_s^0 \ne 0$ and $q_s \in \mathbb{Q}.$ Set $q_{s_*} := \min_{s \in I} q_s$ for some $s_* \in I.$ From Condition (a), we get $q_{s_*}<0$. It is clear that $\|\varphi(t)\| = c t^{q_{s_*}} + o(t^{q_{s_*}})$ as $t \to 0,$ for some $c > 0.$

Recall that
$$\Bbb R^I := \{x\in\Bbb R^n:\ x_s=0 \text{ for all }s\not\in I\}.$$
For $(i,j)\in\{1,\ldots,p\}^2,$ let $d_{ij}$ be the minimal value of the linear function $\sum_{s\in I}q_s\kappa_s$ on
$\Gamma(f_{ij})\cap\Bbb R^I$ and let $\Delta_{ij}$ (resp., $\Delta$) be the unique maximal face of $\Gamma(f_{ij})\cap\Bbb R^I$ (resp., $\Gamma(F)\cap\Bbb R^I)$ where the linear function takes this value. Then a direct computation shows that $\Delta = \sum_{i, j = 1, \ldots, p} \Delta_{ij}.$ Further, since $f_{ii}$ is convenient, $d_{ii} < 0$ and $\Delta_{ii}$ is a face of $\Gamma_\infty(f_{ii}).$ Consequently, we have $\Delta$ is a face of $ \Gamma_\infty(F).$

If we write
$f_{ij}(x) = \sum_{\kappa\in\Gamma(f_{ij})} a_{ij, \kappa} x^\kappa,$ then
\begin{equation}\label{Eq9}
\begin{array}{lll}
$$f_{ij}(\varphi(t))&=\displaystyle\sum_{\kappa\in\Gamma(f_{ij})\cap\Bbb R^I}a_{ij,\kappa}(\varphi(t))^\kappa$$\\
$$&=\displaystyle\sum_{\kappa\in\Gamma(f_{ij})\cap\Bbb R^I} a_{ij,\kappa}(\varphi_1(t))^{\kappa_1}\ldots(\varphi_n(t))^{\kappa_n}$$\\
$$&=\displaystyle\sum_{\kappa\in\Gamma(f_{ij})\cap\Bbb R^I}\left(a_{ij, \kappa} (x_1^0t^{q_1})^{\kappa_1}\ldots(x_n^0t^{q_n})^{\kappa_n}+\textrm{ higher order terms in } t\right)$$\\
$$&=\displaystyle\sum_{\kappa\in\Gamma(f_{ij})\cap\Bbb R^I}\left(a_{ij, \kappa} (x^0)^\kappa t^{\sum_{s \in I}q_s \kappa_s}+\textrm{ higher order terms in } t\right)$$\\
$$&=\displaystyle\sum_{\kappa\in\Delta_{ij}}a_{ij,\kappa} (x^0)^\kappa t^{d_{ij}}+\textrm{ higher order terms in } t$$\\
$$&=\displaystyle f_{ij, \Delta_{ij}}(x^0)t^{d_{ij}} + \textrm{ higher order terms in } t,$$
\end{array}
\end{equation}
where $x^0 := (x_1^0, \ldots, x_n^0)$ with $x_s^0 := 1$ for $s \not \in I.$

Let $J := \{l \in \{1,\ldots,r\}: \ \lambda_l \not \equiv 0 \}.$ Condition (b) implies that $J\ne\emptyset.$ For $l \in J,$ expand the coordinate function $\lambda_l$ in terms of the parameter $t$ as follows
$$\lambda_l(t) =  \lambda_l^0 t^{\theta_l} + \textrm{ higher order terms in } t,$$
where $\lambda_l^0 > 0$ and $\theta_l\ge 0.$

For $l=1,\ldots,r$, let $K_l := \{i \in \{1,\ldots,p\}: \ v_i^l\not \equiv 0 \}.$ By Condition (c), $K_l \ne\emptyset$. For $i\in K_l$, expand the coordinate function $v_i^l$ in terms of the parameter $t$ as follows
$$v_i^l(t)=w_i^l t^{\mu_i^l} + \textrm{ higher order terms in } t,$$
where $w_i^l \ne 0$ and $\mu_i^l\ge 0.$

For simplicity, let
$$u_s(t) := \sum_{l = 1}^r \lambda_l(t) \sum_{i, j = 1}^p v_i^l(t)v_j^l(t)\frac{\partial f_{ij}}{\partial x_s}(\varphi(t)) \quad \textrm{ for } \quad s = 1, \ldots, n.$$
We have for all $s \in I,$
$$\frac{\partial f_{ij}}{\partial x_s}(\varphi(t))= \frac{\partial f_{ij, \Delta_{ij}}}{\partial x_s}(x^0)t^{d_{ij} - q_s}+ \textrm{ higher order terms in } t,$$
and hence
\begin{eqnarray*}
u_s(t) &=& \displaystyle\sum_{l\in J}\lambda_l(t)\sum_{i,j\in K_l}v_i^l(t)v_j^l(t)\frac{\partial f_{ij}}{\partial x_s}(\varphi(t))\\
&=&\displaystyle\sum_{l\in J}\sum_{i,j\in K_l}\left(\lambda_l^0w_i^lw_j^l\frac{\partial f_{ij,\Delta_{ij}}}{\partial x_s}(x^0)t^{d_{ij}+\theta_l+\mu_i^l+\mu_j^l-q_s}+ \textrm{ higher order terms in } t\right)\\
&=&\displaystyle\left(\sum_{l\in J'}\sum_{(i,j)\in L_l }\lambda_l^0w_i^lw_j^l\frac{\partial f_{ij,\Delta_{ij}}}{\partial x_s}(x^0)\right)t^{M-q_s}+ \textrm{ higher order terms in } t,
\end{eqnarray*}
where we put
\begin{eqnarray*}
M &:=& \displaystyle\min_{l\in J,\ i,j\in K_l} d_{ij}+\theta_l+\mu_i^l+\mu_j^l, \\
J' &:=& \{l \in J:\ \exists i,j\in K_l \text{ s.t }d_{ij}+\theta_l+\mu_i^l+\mu_j^l=M\}\ne\emptyset, \\
L_l &:=& \{(i,j) \in K_l \times K_l : \ d_{ij}+\theta_l+\mu_i^l+\mu_j^l=M \}.
\end{eqnarray*}

There are two cases to be considered.

\subsubsection*{Case 1: $M  \le q_{s_*}  := \min_{s \in I} q_s$} \

For $s \in I,$ we have $M - q_s \le M - q_{s_*} \le 0.$ Then it follows from Condition (d) that
$$\sum_{l\in J'}\sum_{(i,j)\in L_l }\lambda_l^0w_i^lw_j^l\frac{\partial f_{ij,\Delta_{ij}}}{\partial x_s}(x^0) = 0.$$
For $s\not\in I$, $f_{ij,\Delta_{ij}}$ does not depend on $x_s$, so $\frac{\partial f_{ij,\Delta_{ij}}}{\partial x_s}\equiv 0$. Therefore, for $s=1,\ldots,n,$
\begin{equation*}
\sum_{l\in J'}\sum_{(i,j) \in L_l}\lambda_l^0w_i^lw_j^l\frac{\partial f_{ij,\Delta_{ij}}}{\partial x_s}(x^0) \ = \ 0.
\end{equation*}
Set $\Omega := (\omega_{ij})_{i,j=1,\ldots,p}$ with
$$\omega_{ij} :=
\begin{cases}
\sum_{l\in J':\ (i,j)\in L_l}\lambda_l^0w_i^lw_j^l & \text{ if  }\ \exists l\in J' \text{ s.t } (i,j)\in L_l, \\
0&\text{ if } \not\exists l\in J' \text{ s.t } (i,j)\in L_l.
\end{cases}$$
Since $(i,j)\in L_l$ if and only if $(j,i)\in L_l$, the matrix $\Omega$ is symmetric. Further, for $s=1,\ldots,n$,
$$\tr\left(\Omega\frac{\partial F_\Delta}{\partial x_s}(x^0)\right)= \sum_{i,j=1}^p\omega_{ij}\frac{\partial f_{ij,\Delta_{ij}}}{\partial x_s}(x^0) = 0.$$

Let $(i,j)\in L_l$. It follows from the assumption $\Gamma(f_{ij})\subseteq \Gamma(f_{ii})$ that $d_{ij}\ge d_{ii}.$ By symmetry, we also have $\Gamma(f_{ij})\subseteq \Gamma(f_{jj})$ and so $d_{ij}\ge d_{jj}.$ Assume that $\mu_i^l\le \mu_j^l$, then
\begin{equation}\label{Eq10}
d_{ii} + \theta_l + \mu_i^l + \mu_i^l\le d_{ij} + \theta_l + \mu_i^l + \mu_j^l = M.
\end{equation}
By the definition of $M$, the inequality in (\ref{Eq10}) must be equality which implies that $d_{ii}=d_{ij}$ and $\mu_i^l= \mu_j^l$. Again, by symmetry, we get $d_{jj}=d_{ij}=d_{ii}.$ Finally, it follows that $(i,i),(j,j)\in L_l.$ Now, there exist indexes $l_0, i_0,$ and $j_0$ such that $l_0 \in J'$ and $(i_0,j_0)\in L_{l_0}$. So $(i_0,i_0)\in L_{l_0}$ and
\begin{equation*}
\omega_{i_0i_0} \ = \ \sum_{l\in J':\ (i_0,i_0)\in L_l}\lambda_l^0 w_{i_0}^lw_{i_0}^l
\ = \ \sum_{l\in J':\ (i_0,i_0)\in L_l}\lambda_l^0 (w_{i_0}^l)^2 \ > \ 0.
\end{equation*}
Moreover, by definition, each nonzero element on the diagonal of $\Omega$ is positive. Hence $\tr(\Omega)>0.$

By Remark \ref{Relax} and by the assumption of non-degeneracy at infinity of $F$, we get
\begin{equation}\label{Eq11}
\displaystyle\sum_{i,j=1}^p\left(\sum_{l\in J':\ (i,j)\in L_l}\lambda_l^0w_i^lw_j^l\right)f_{ij,\Delta_{ij}}(x^0)=\sum_{i,j=1}^p\omega_{ij}f_{ij,\Delta_{ij}}(x^0)=\tr(\Omega F_\Delta(x^0))\ne 0.
\end{equation}
From Conditions (b), (c), and (\ref{Eq9}) we have
\begin{eqnarray*}
f(\varphi(t))&=&\displaystyle\sum_{l\in J}\lambda_l(t)\langle F(\varphi(t))v^l(t),v^l(t)\rangle\\
&=&\displaystyle\sum_{l\in J}\sum_{i,j=1}^p\lambda_l(t)v_i^l(t)v_j^l(t)f_{ij}(\varphi(t))\\
&=&\displaystyle\sum_{l\in J}\sum_{i,j\in K_l}(\lambda_l^0t^{\theta_l} + \cdots) (w_i^lt^{\mu_i^l}  + \cdots )(w_j^lt^{\mu_j^l}  + \cdots)(f_{ij,\Delta_{ij}}(x^0)t^{d_{ij}}  + \cdots)\\
&=&\displaystyle\sum_{l\in J}\sum_{i,j\in K_l}(\lambda_l^0w_i^lw_j^lf_{ij,\Delta_{ij}}(x^0))t^{d_{ij}+\theta_l+\mu_i^l+\mu_j^l}  + \cdots\\
&=&\displaystyle\sum_{l\in J'}\sum_{(i,j)\in L_l}(\lambda_l^0w_i^lw_j^lf_{ij,\Delta_{ij}}(x^0))t^{M}  + \cdots\\
&=&\displaystyle\sum_{i,j=1}^p\left(\sum_{l\in J':\ (i,j)\in L_l}\lambda_l^0w_i^lw_j^l\right)f_{ij,\Delta_{ij}}(x^0)t^{M}  + \cdots \\
&=&\displaystyle\tr(\Omega F_\Delta(x^0))t^{M} + \textrm{ higher order term in } t.
\end{eqnarray*}
By Inequality  (\ref{Eq11}), we have\footnote{We say that $a(t) \simeq b(t)$ as $t \to 0$ if there exist positive constants $c_1$ and $c_2$ such  that $c_1|a(t)| \le |b(t)| \le c_2 |a(t)|$ for $0 \le t \ll 1.$} $f(\varphi(t))\simeq t^M$ as $t \to 0$, so
\begin{equation} \label{Eq12}
\left|\frac{d (f \circ \varphi)(t)}{dt}\right| \simeq t^{M-1} \quad \textrm{ as } \quad t \to 0.
\end{equation}
On the other hand, we have
\begin{eqnarray*}
\displaystyle\frac{d (f \circ \varphi) (t)}{dt}&=&\displaystyle\frac{d}{dt}\left(\sum_{l\in J}\lambda_l(t)\langle F(\varphi(t))v^l(t),v^l(t)\rangle\right) \\
&=&\displaystyle\sum_{l\in J}\frac{d\lambda_l(t)}{dt}\langle F(\varphi(t))v^l(t),v^l(t)\rangle+\sum_{l\in J}\lambda_l(t)\frac{d\langle F(\varphi(t))v^l(t),v^l(t)\rangle}{dt}.
\end{eqnarray*}
Since
\begin{eqnarray*}
\sum_{l\in J}\frac{d\lambda_l(t)}{dt}\langle F(\varphi(t))v^l(t),v^l(t)\rangle &=& \displaystyle\sum_{l\in J}\frac{d\lambda_l(t)}{dt}f(\varphi(t)) \\
&=& \frac{d\left(\sum_{l\in J}\lambda_l(t)\right)}{dt}f(\varphi(t)) = \frac{d1}{dt}f(\varphi(t))=0,
\end{eqnarray*}
we see that
\begin{displaymath}
\begin{array}{lll}
$$\displaystyle\frac{d (f \circ \varphi)(t)}{dt}&=&\displaystyle\sum_{l\in J}\lambda_l(t)\frac{d\langle F(\varphi(t))v^l(t),v^l(t)\rangle}{dt}$$\\
$$&=&\displaystyle\sum_{l\in J}\lambda_l(t)\frac{d}{dt}\left(\sum_{i,j=1}^pv_i^l(t)v_j^l(t)f_{ij}(\varphi(t))\right)$$\\
$$&=&\displaystyle\sum_{l\in J}\lambda_l(t)\sum_{i,j=1}^p\left(\frac{dv_i^l(t)}{dt}v_j^l(t)f_{ij}(\varphi(t))+v_i^l(t)\frac{dv_j^l(t)}{dt}f_{ij}(\varphi(t))\right)$$\\
$$&& + \displaystyle\sum_{l\in J}\lambda_l(t)\sum_{i,j=1}^pv_i^l(t)v_j^l(t)\langle\nabla f_{ij}(\varphi(t)), \frac{d \varphi}{d t} \rangle$$\\
$$&=& \displaystyle\sum_{l\in J}\lambda_l(t) \left(\sum_{i=1}^p\frac{dv_i^l(t)}{dt}\sum_{j=1}^pv_j^l(t)f_{ij}(\varphi(t))+\sum_{j=1}^p\frac{dv_j^l(t)}{dt}\sum_{i=1}^pv_i^l(t)f_{ij}(\varphi(t))\right)$$\\
$$&& + \displaystyle\langle u(t), \frac{d \varphi}{d t}  \rangle.$$
\end{array}
\end{displaymath}
Note that $F$ is symmetric and $F(\varphi(t))v^l=f(\varphi(t))v^l$, so
$$\displaystyle\sum_{j=1}^pv_j^l(t)f_{ij}(\varphi(t))=f(\varphi(t))v_i^l(t)\ \text{ and } \ \sum_{i=1}^pv_i^l(t)f_{ij}(\varphi(t))=f(\varphi(t))v_j^l(t).$$
Thus
\begin{displaymath}
\begin{array}{lll}
$$\displaystyle\frac{d (f \circ \varphi)(t)}{dt}&=&\displaystyle\sum_{l\in J}\lambda_l(t) \left(\sum_{i=1}^p\frac{dv_i^l(t)}{dt}f(\varphi(t))v_i^l(t)+\sum_{j=1}^p\frac{dv_j^l(t)}{dt}f(\varphi(t))v_j^l(t)\right)+\langle u(t), \frac{d \varphi}{d t} \rangle$$\\
$$&=&\displaystyle\frac{1}{2}\sum_{l\in J}\lambda_l(t)f(\varphi(t)) \left(\sum_{i=1}^p\frac{d(v_i^l(t))^2}{dt}+\sum_{j=1}^p\frac{d(v_j^l(t))^2}{dt}\right)+\langle u(t), \frac{d \varphi}{d t} \rangle$$\\
$$&=&\displaystyle\sum_{l\in J}\lambda_l(t)f(\varphi(t)) \frac{d}{dt}\|v^l(t)\|^2+\langle u(t), \frac{d \varphi}{d t} \rangle$$\\
$$&=&\displaystyle\sum_{l\in J}\lambda_l(t)f(\varphi(t)) \frac{d1}{dt}+\langle u(t), \frac{d \varphi}{d t}  \rangle$$\\
$$&=&\langle u(t), \frac{d \varphi}{d t} \rangle\le \|u(t)\| \| \frac{d \varphi}{d t}  \|=\frak m_f(\varphi(t))\|\frac{d \varphi}{d t} \|.$$\\
\end{array}
\end{displaymath}
This, together with Inequality (\ref{Eq12}), implies that there is $c' > 0$ such that
$$\frak m_f(\varphi(t))\ge \frac{c't^{M-1}}{t^{q_{s_*}-1}}=c't^{M-q_{s_*}}.$$
Since $\frak m_f(\varphi(t))\to 0$, it follows that $M-q_{s_*}>0,$ which is a contradiction.

\subsubsection*{Case 2: $M > q_{s_*}  := \min_{s \in I} q_s$} \

Recall that $\theta_l \ge 0$ and $\mu_i^l\ge 0$ for all $l\in J$ and $i\in K_l$. By Conditions (b) and (c), $\theta_{l_0}=0$ and $\mu_{i_0}^{l_0}= 0$ for some $l_0\in J$ and $i_0\in K_{l_0}.$ Since $f_{i_0i_0}$ is convenient, for $s = 1, \ldots, n,$ there exists an integer $m_s \ge 1$ such that $m_s e_s \in \Gamma_\infty(f_{i_0 i_0}).$ Then it is clear that
$$q_s m_s \ge d_{i_0 i_0}, \quad \textrm{ for all } \quad s \in I.$$
On the other hand, we have
$$d_{i_0 i_0} = d_{i_0 i_0} + \theta_{l_0} + \mu_{i_0}^{l_0} + \mu_{i_0}^{l_0} \ge \min_{l\in J,\ (i,j)\in K_l}(d_{ij}+\theta_l+\mu_i^l+\mu_j^l)=M.$$
Therefore
$$q_{s_*} m_{s_*} \ge d_{i_0 i_0} \ge M  > q_{s_*}.$$
Since $q_{s_*}  = \min_{s \in I} q_s   < 0,$ it implies that $m_{s_*} < 1$, which is a contradiction.
\end{proof}

\begin{lemma}\label{Lemma6}
Assume that there exist some constants $c > 0$ and $R > 0$ such that
\begin{equation*}
{\frak m}_{f} (x) \ge c \quad \textrm{ for all } \quad x \in f^{-1}((0, +\infty)) \quad \textrm{ and } \quad \|x\| \ge R.
\end{equation*}
Let $s \in S_F.$ Then
\begin{equation*}
\frac{c}{2}\mathrm{dist}(x, S_F)\le [f(x)]_+ \quad \textrm{ for all } \quad \|x\|\ge 3R + 2\|s\|.
\end{equation*}
\end{lemma}
\begin{proof}
We argue by contradiction. Suppose that the conclusion is false. Then there exists $\bar x \in \mathbb{R}^n$ such that $$\|\bar x\|\ge 3R + 2\|s\| \quad \textrm{ and } \quad [f(\bar x)]_+<\frac{c}{2}\mathrm{dist}(\bar x, S_F).$$
Clearly $\bar x\not\in S_F$. Set $K:=\{x \in \mathbb{R}^n \ : \ \|x\|\ge R\}.$ Note that $\inf_K [f(x)]_+\ge 0$, and so
$$[f(\bar{x})]_+ < \inf_K [f(x)]_+ + \frac{c}{2}\mathrm{dist}(\bar{x}, S_F).$$
By applying Ekeland variational principle \cite{Ekeland1979} to the function $[f(x)]_+$ on the closed set $K$ with the data $\epsilon := \frac{c}{2}\mathrm{dist}(\bar x, S_F) > 0$ and $\lambda := \frac{2\mathrm{dist}(\bar x, S_F)}{3} > 0$, there is $\bar y\in K$ such that $\|\bar y - \bar x\| < {\lambda}$ and that $\bar y$ minimizes the function
$$x\mapsto [f(x)]_+ + \frac{\epsilon}{\lambda}\|x-\bar y\|.$$
It follows that
\begin{eqnarray*}
\|\bar y\| & \ge & \|\bar x\| - \|\bar y - \bar x\| >  \|\bar x\| - \frac{2}{3}\mathrm{dist}(x, S_F) \\
& \ge & \|\bar{x}\| - \frac{2}{3}\|\bar{x} - s\| \ge \|\bar{x}\| - \frac{2}{3}(\|\bar x\| + s\|) 
= \frac{1}{3}(\|\bar{x} - 2 \|s\|) \ge R.
\end{eqnarray*}
Thus $\bar y$ is an interior point of $K.$ Then we deduce from \cite[Theorem~5.21(iii)]{Mordukhovich2006} that
$$0\in \partial [f(\bar y)]_+ + \frac{\epsilon}{\lambda}\Bbb B^n.$$
By the definition of the function $\frak m_{f_+},$ it follows easily that
$$\frak m_{f_+} (\bar y)\leq \frac{\epsilon}{\lambda}.$$
Since $\bar x \not \in S_F$ and $\|\bar y - \bar x\| < \lambda = \frac{2}{3}\mathrm{dist}(\bar x, S_F)$, we have 
$\bar{y} \not \in S_F$ and so $f(\bar y)>0.$ Therefore
$$\frak m_{f} (\bar y)=\frak m_{f_+} (\bar y)\leq \frac{\epsilon}{\lambda} = \frac{3c}{4} < c,$$
which is a contradiction.
\end{proof}

Now, we are in position to finish the proof of Theorem \ref{HolderTypeTheorem}.

\begin{proof} [Proof of Theorem \ref{HolderTypeTheorem}] 
By Lemma \ref{Lemma5}, there exist some constants $c_1 > 0$ and $R > 0$ such that
\begin{equation*}
{\frak m}_{f} (x) \ge c_1 \quad \textrm{ for all } \quad \|x\| \ge R.
\end{equation*}
Let us fix a point $s$ in $S_F.$ Due to Lemma \ref{Lemma6}, we obtain
\begin{equation}\label{nu1}
\frac{c_1}{2} \mathrm{dist}(x, S_F) \le [f(x)]_+, \quad \text{ for all } \quad  \|x\|\ge 3R + 2\|s\|.
\end{equation}

On the other hand, thanks to Theorem \ref{CompactErrorBound}, we get a constant $c_2 > 0$ satisfying
\begin{equation}\label{nu2}
c_2 \mathrm{dist}(x, S_F)\leq [f(x)]_+^{\frac{1}{\mathscr{R}(2n+p(n+1),d+3)}}, \quad \textrm{ for all }  \ \|x\| \le 3R + 2\|s\|.
\end{equation}
Let $c := \min \{\frac{c_1}{2}, c_2\} > 0.$  Taking account of   (\ref{nu1}) and  (\ref{nu2}), we obtain
$$c\mathrm{dist}(x,S_F)\leq [f(x)]_+^{\frac{1}{\mathscr{R}(2n+p(n+1),d+3)}} +[f(x)]_+, \quad \textrm{for all } \quad x \in \Bbb R^n,$$
as it was to be shown.
\end{proof}

\subsection*{Acknowledgments}
We are grateful to the referee for careful reading and corrections of the manuscript.

\end{document}